\documentclass{amsart}
\usepackage[margin=3.5cm]{geometry}
\usepackage{xcolor}
\usepackage{amsfonts,amsthm,amssymb,verbatim,amscd,amsmath,blindtext}
\usepackage{multirow}
\usepackage{graphicx}
\usepackage{enumitem}
\usepackage{amssymb}
\usepackage{ytableau}
\usepackage{tikz}
\usepackage{tikz-cd}
\usepackage{float,pifont}
\usepackage{mathtools}
\usepackage[pagebackref,colorlinks=true,citecolor=blue]{hyperref}

\newtheorem{theorem}{Theorem}[section]
\newtheorem{lemma}[theorem]{Lemma}
\newtheorem{corollary}[theorem]{Corollary}
\newtheorem{proposition}[theorem]{Proposition}
\newtheorem{claim}[theorem]{Claim}

\theoremstyle{definition}
\newtheorem{definition}[theorem]{Definition}
\newtheorem{problem}[theorem]{Problem}
\newtheorem{example}[theorem]{Example}
\newtheorem{remark}[theorem]{Remark}
\newtheorem{question}[theorem]{Question}

\makeatletter
\newenvironment{sqcases}{%
  \matrix@check\sqcases\env@sqcases
}{%
  \endarray\right.%
}
\def\env@sqcases{%
  \let\@ifnextchar\new@ifnextchar
  \left\lbrack
  \def\arraystretch{1.2}%
  \array{@{}l@{\quad}l@{}}%
}
\makeatother

\DeclareMathOperator{\lcm}{lcm}
\DeclareMathOperator{\sbridge}{sb}
\DeclareMathOperator{\mingens}{Gens}
\DeclareMathOperator{\BF}{BF}
\DeclareMathOperator{\w}{w}
\DeclareMathOperator{\dist}{dist}

\renewcommand{\H}{\mathcal{H}}

\newcommand{\ZZ}{{\mathbb Z}}
\newcommand{\NN}{{\mathbb N}}

\newcommand{\card}[1]{{\left\lvert #1 \right\lvert}}

\newcommand{\Kfoursymb}[1][1]{%
    \begin{tikzpicture}[scale=#1, %
    baseline=-0.3ex,
    thick,
    ]
    \def\x{2.5mm};
    \def\r{0.2mm};
    \draw (0,0) -- (\x, 0) -- (\x, \x) -- (0, \x) -- cycle;
    \draw (0, 0) -- (\x, \x);
    \draw (\x, 0) -- (0, \x);

    \draw[fill=black] (0,0) circle (\r);
    \draw[fill=black] (0,\x) circle (\r);
    \draw[fill=black] (\x,0) circle (\r);
    \draw[fill=black] (\x,\x) circle (\r);
    \end{tikzpicture}%
}

\newcommand{\diamondsymb}[1][1]{%
    \begin{tikzpicture}[scale=#1, %
    baseline=-0.3ex,
    thick,
    ]
    \def\x{2.5mm};
    \def\r{0.11mm};
    \draw (0,0) -- (\x, 0) -- (\x, \x) -- (0, \x) -- cycle;
    \draw (0, 0) -- (\x, \x);

    \draw[fill=black] (0,0) circle (\r);
    \draw[fill=black] (0,\x) circle (\r);
    \draw[fill=black] (\x,0) circle (\r);
    \draw[fill=black] (\x,\x) circle (\r);
    \end{tikzpicture}%
}

\newcommand{\gemsymb}[1][1]{%
    \begin{tikzpicture}[scale=#1, %
        thick,
    baseline=-0.25ex
    ]
    \def\x{2.5mm};
    \def\r{0.11mm};
    \draw (0,0) -- (\x/2, \x) -- (\x, 0) -- (3*\x/2, \x) -- (2*\x, 0) -- cycle;
    \draw (\x/2, \x) -- (3*\x/2, \x);

    \draw[fill=black] (0,0) circle (\r);
    \draw[fill=black] (\x/2, \x) circle (\r);
    \draw[fill=black] (\x,0) circle (\r);
    \draw[fill=black] (3*\x/2, \x) circle (\r);
    \draw[fill=black] (2*\x, 0) circle (\r);
    \end{tikzpicture}%
}

\newcommand{\kitesymb}[1][1]{%
    \begin{tikzpicture}[scale=#1, %
        thick,
    baseline=-1ex
    ]
    \def\x{2.5mm};
    \def\r{0.11mm};
    \draw (0,0) -- (\x/2, \x/2) -- (\x, 0) -- (\x/2, -\x/2) -- cycle;
    \draw (\x/2, \x/2) -- (\x/2, -\x/2);
    \draw (\x,0) -- (3*\x/2,0);

    \draw[fill=black] (0,0) circle (\r);
    \draw[fill=black] (\x/2, \x/2) circle (\r);
    \draw[fill=black] (\x,0) circle (\r);
    \draw[fill=black] (\x/2, -\x/2) circle (\r);
    \draw[fill=black] (3*\x/2,0) circle (\r);
    \end{tikzpicture}%
}

\newcommand{\pawsymb}[1][1]{%
    \begin{tikzpicture}[scale=#1, %
        thick,
    baseline=-1ex
    ]
    \def\x{2.5mm};
    \def\r{0.11mm};
    \draw (\x/2, \x/2) -- (\x, 0) -- (\x/2, -\x/2) -- cycle;
    \draw (\x,0) -- (3*\x/2,0);

    \draw[fill=black] (\x/2, \x/2) circle (\r);
    \draw[fill=black] (\x, 0) circle (\r);
    \draw[fill=black] (\x/2, -\x/2) circle (\r);
    \draw[fill=black] (3*\x/2,0) circle (\r);
    \end{tikzpicture}%
}

\newcommand{\tadpolesymb}[1][1]{%
    \begin{tikzpicture}[scale=#1, %
				thick,
		baseline=-1ex
		]
		\def\x{3mm};
		\def\r{0.11mm};
		\draw (\x/2, \x/2) -- (\x, 0) -- (\x/2, -\x/2) -- cycle;
		\draw (\x,0) -- (3 * \x/2, 0) -- (4*\x/2, -\x/2);

        \draw[fill=black] (\x/2, \x/2) circle (\r);
        \draw[fill=black] (\x, 0) circle (\r);
        \draw[fill=black] (\x/2, -\x/2) circle (\r);
        \draw[fill=black] (3*\x/2,0) circle (\r);
        \draw[fill=black] (4*\x/2, -\x/2) circle (\r);
	\end{tikzpicture}%
}

\newcommand{\butterflysymb}[1][1]{%
    \begin{tikzpicture}[scale=#1, %
        thick,
    baseline=-1ex
    ]
    \def\x{3mm};
    \def\r{0.11mm};
    \draw (\x/2, \x/2) -- (\x, 0) -- (\x/2, -\x/2) -- cycle;
    \draw (3*\x/2, \x/2) -- (\x, 0) -- (3*\x/2, -\x/2) -- cycle;
    \draw[fill=black] (\x/2, \x/2) circle (\r);
    \draw[fill=black] (\x, 0) circle (\r);
    \draw[fill=black] (\x/2, -\x/2) circle (\r);
    \draw[fill=black] (3*\x/2, \x/2) circle (\r);
    \draw[fill=black] (3*\x/2, -\x/2) circle (\r);
    \end{tikzpicture}%
}

\newcommand{\sunletsymb}[2][1]{%
    \begin{tikzpicture}[scale=#1, %
        thick,
    baseline=-1ex
    ]
    \def\x{1mm};
    \def\r{0.11mm};
    \def\theta{360/#2};

    \foreach \i in {1,...,#2}{
        \draw ({\x*cos(\theta*\i)}, {\x*sin(\theta*\i)}) -- ({\x*cos(\theta*(\i+1))}, {\x*sin(\theta*(\i+1))});
        \draw ({\x*cos(\theta*\i)}, {\x*sin(\theta*\i)}) -- ({2*\x*cos(\theta*\i)}, {2*\x*sin(\theta*\i)});
        
        \draw[fill=black] ({\x*cos(\theta*\i)}, {\x*sin(\theta*\i)}) circle (\r);
        \draw[fill=black] ({2*\x*cos(\theta*\i)}, {2*\x*sin(\theta*\i)}) circle (\r);
    }
    \end{tikzpicture}%
}

\newcommand{\smallestcounterexamplesymb}{%
	\begin{tikzpicture}[scale=0.8, inner sep=0pt, baseline=0ex]
		\def\x{7mm};
		\def\r{0.45mm};
		
	    \node (t) at (-2*\x, 0) {};
	    \node (u) at (0, \x) {};
	    \node (v) at (\x, 0) {};
	    \node (w) at (0, 0) {};
	    \node (x) at (-\x, 0) {};
	    \node (y) at (0, -\x) {};
	    \node (z) at (2*\x, 0) {};

	    \draw[] (t) to (x);
	    \draw[] (x) to (w);
	    \draw[] (u) to (v);
	    \draw[] (u) to (x);
	    \draw[] (v) to (w);
	    \draw[] (y) to (v);
	    \draw[] (y) to (x);
	    \draw[] (v) to (z);

	    \draw[fill=black] (t) circle (\r);
	    \draw[fill=black] (u) circle (\r);
	    \draw[fill=black] (v) circle (\r);
	    \draw[fill=black] (w) circle (\r);
	    \draw[fill=black] (x) circle (\r);
	    \draw[fill=black] (y) circle (\r);
	    \draw[fill=black] (z) circle (\r);
	\end{tikzpicture}%
}

\newcommand{\joinsixcyclessymb}{%
\begin{tikzpicture}[scale=0.8,every node/.style={draw=black,circle}]
		\def\x{15mm};
		\def\r{0.8mm};
		
		\node[] (s) at (0, 1.5*\x) {$s$};
        \node[] (t) at (-1*\x, 0.5*\x) {$t$};
        \node[] (u) at (0, 0.5*\x) {$u$};
        \node[] (v) at (1*\x, 0.5*\x) {$v$};
        \node[] (w) at (-1*\x, -0.5*\x) {$w$};
        \node[] (x) at (0, -0.5*\x) {$x$};
        \node[] (y) at (1*\x, -0.5*\x) {$y$};
        \node[] (z) at (0, -1.5*\x) {$z$};

        \draw[] (s) to (t);
        \draw[] (s) to (u);
        \draw[] (s) to (v);
        \draw[] (t) to (w);
        \draw[] (u) to (x);
        \draw[] (v) to (y);
        \draw[] (w) to (z);
        \draw[] (x) to (z);
        \draw[] (y) to (z);
	\end{tikzpicture}%
}

\newcommand{\cricketsymb}{%
\begin{tikzpicture}[scale=1, %
		baseline=1ex,
		thick,
		]
		\def\x{15mm};
		\def\r{0.3mm};

        \node[label={[xshift=-0.8em, yshift=-1em] \small $x$}] (x) at (0, 0) {};
        \node[label={[xshift=0.8em, yshift=-1em] \small $y$}] (y) at (\x, 0) {};
        \node[label={[xshift=-0.8em, yshift=-1em] \small $z$}] (z) at (0.5 * \x, 0.8 * \x) {};

        \node[label={[xshift=0em, yshift=-1.3em] \small $\cdots$}] (dots) at (0.5 * \x, 0.8 * \x + 0.5 * \x) {};
  
		\draw (0, 0) -- (0.5 * \x, 0.8 * \x) -- (\x, 0) -- cycle;

        \foreach \i in {-1, 1}{
            \foreach \j in {0.5, 0.7}{
                 \draw (0.5 * \x, 0.8 * \x) -- (0.5 * \x + \i * \j * \x, 0.8 * \x + 0.5 * \x);
                 \draw[fill=black] (0.5 * \x + \i * \j * \x, 0.8 * \x + 0.5 * \x) circle (\r);
            }
        }
  
		\draw[fill=black] (x) circle (\r);
		\draw[fill=black] (y) circle (\r);
		\draw[fill=black] (z) circle (\r);
	\end{tikzpicture}%
}

\begin{document}
\title{Minimal Cellular Resolutions of Powers of Graphs}
\author{Trung Chau}
\address{Department of Mathematics, University of Utah, 155 South 1400 East, Salt Lake City, UT~84112, USA}
\email{trung.chau@utah.edu}

\author{T\`ai Huy H\`a}
\address{Tulane University, Mathematics Department, 6823 St. Charles Avenue, New Orleans, LA 70118, USA}
\email{tha@tulane.edu}

\author{Aryaman Maithani}
\address{Department of Mathematics, University of Utah, 155 South 1400 East, Salt Lake City, UT~84112, USA}
\email{maithani@math.utah.edu}

\keywords{Lyubeznik resolution, Barile-Macchia resolution, cellular resolution, edge ideal, graph, hypergraph, tree, hypertree, powers of edge ideals, free resolution, monomial ideal}

\subjclass[2020]{13D02; 13F55; 05C65; 05C75; 05E40}

\begin{abstract}
   Let $G$ be a connected graph and let $I(G)$ denote its edge ideal. We classify when $I(G)^n$, for $n \ge 1$, admits a minimal Lyubeznik resolution. We also give a characterization for when $I(G)^n$ is bridge-friendly, which, in turn, implies that $I(G)^n$ has a minimal Barile-Macchia cellular resolution.
\end{abstract}

\maketitle


\section{Introduction}

It has been a central problem, in the study of minimal free resolutions, to understand when a monomial ideal admits cellular resolutions and how to construct these resolutions (cf. \cite{AFG2020, BM20, BW02, BPS98, BS98, CK24, CT2016, CEFMMSS21, CEFMMSS22, Ly88, OY2015, Vel08}). Despite much effort from many researchers, only a few explicit constructions of simplicial complexes that support the free resolution of a monomial ideal in general are known. The resolutions resulted from these explicit constructions are the \emph{Taylor resolution}, \emph{Lyubeznik resolution} and, in special cases, the \emph{Scarf complex} (see \cite{BPS98, Ly88, Tay66}). From discrete Morse theory, cellular resolutions for monomial ideals were constructed in \cite{BM20, BW02, CK24, CEFMMSS21, CEFMMSS22}; particularly, a subclass  considered in \cite{BM20,CK24,CKW24} is called the \emph{Barile-Macchia resolution}. The Taylor resolution is almost never minimal and the Scarf complex is often not a resolution. Classifying monomials ideals whose Lyubeznik or Barile-Macchia resolutions are minimal, or whose Scarf complexes are resolutions, seems out of reach at this time.

In a recent work \cite{FHHM24} of Faridi, H\`a, Hibi, and Morey, graphs whose edge ideals and their powers admit Scarf resolutions are identified.
More precisely, let $G = (V,E)$ be a simple undirected graph (i.e., $G$ contains neither loops nor multiple edges) over the vertex set $V = \{x_1, \dots, x_r\}$. We will always consider \emph{connected} simple graphs with at least one edge. Let $\Bbbk$ be a field and, by identifying the vertices of $G$ with variables, let $S = \Bbbk[x_1, \dots, x_r] = \Bbbk[V]$ be a polynomial ring. The \emph{edge ideal} of $G$ is defined to be
\begin{equation*} 
	I(G) = \left\langle x_ix_j ~\middle|~ \{x_i, x_j\} \in E \right\rangle \subseteq S.
\end{equation*}
It was proved in \cite[Theorem 8.3]{FHHM24} that the Scarf complex of $I(G)^n$, for a connected graph $G$, is a resolution (which is necessarily minimal) if and only if either
\begin{enumerate}
    \item[(S1)] $n = 1$ and $G$ is a \emph{gap-free} tree; or
    \item[(S2)] $n > 1$ and $G$ is either an isolated vertex, an edge, or a path of length 2.
\end{enumerate}
Our work in this paper addresses the following problem:
\begin{problem} \label{prob:main}
Characterize graphs whose edge ideals and their powers have minimal Lyubeznik and/or Barile-Macchia resolutions.
\end{problem}

Our results give a complete classification of graphs whose edge ideals and their powers have minimal Lyubeznik resolutions. Our method is based on the observation that having a minimal Lyubeznik resolution  descends to HHZ-subideals with respect to any given monomial $m$, i.e., ideals generated by subcollections of the generators which divide the given monomial $m$ (see Lemma \ref{lem:HHZ-BM-Lyubeznik}). HHZ-subideals were introduced in \cite{HHZ04}, and appeared briefly prior in \cite{BPS98}. As a consequence, if $H$ is an induced subgraph of $G$ and $I(H)^n$ does not have a minimal Lyubeznik resolution, then neither does $I(G)^n$. This allows us to reduce the problem to finding ``forbidden structures'' for having a minimal Lyubeznik resolution. 

The problem is considerably more difficult with Barile-Macchia resolutions. Even though having a minimal Barile-Macchia resolution also descends to HHZ-subideals (see Lemma \ref{lem:HHZ-BM-Lyubeznik}), we have not been able to find ``forbidden strictures'' for this property. It is too computationally expensive to verify this property even for all graphs with 9 vertices or fewer. 

\begin{remark}
	One known example of a graph whose edge ideal does not have a minimal Barile-Macchia resolution is the 9-cycle \cite[Remark 4.24]{CK24}. The smallest example was found using an exhaustive search on \texttt{Sage} and is drawn here. 

	\begin{center}
		\smallestcounterexamplesymb
	\end{center}
\end{remark}

Chau and Kara \cite{CK24} introduced the notion of \emph{bridge-friendly} monomial ideals (see Definition \ref{def:bridgefriendly}), which implies that the given monomial ideal has a minimal Barile-Macchia resolution, and this resolution can be nicely described. We will classify \emph{chordal} graphs whose edge ideals are bridge-friendly and, thus, admit minimal cellular resolutions.

Let us now describe our main results in more details. For nonnegative integers $a,b,c$, let $L(a,b,c)$ denote the graph consisting of exactly $c$ triangles sharing one edge $\{x,y\}$ such that $a$ and $b$ distinct leaves are attached to the vertices $x$ and $y$, respectively (see Definition \ref{def:Labc}). On the other hand, associated to a tree $T$ and an edge-weight function $\w$ on the edges of $T$, let $\text{BF}(T,\w)$ be the graph obtained by attaching $\w(e)$ triangles to each edge $e$ in $T$ (see Definition \ref{def:BF}). For the edge ideal $I(G)$ itself, we establish in Theorems \ref{thm:Lyubeznik-graphs} and \ref{thm:bridge-friendly-chordal} the following results.

\begin{enumerate}
\item[(\textbf{\ref{thm:Lyubeznik-graphs}})] $I(G)$ has a minimal Lyubeznik resolution if and only if $G = L(a,b,c)$ for some nonnegative integers $a,b,c$; and
\item[(\textbf{\ref{thm:bridge-friendly-chordal}})] if $G$ is a chordal graph, then $I(G)$ is bridge-friendly if and only if $G =\text{BF}(T,\w),$ for some tree $T$ and edge-weight function $\w$ on $T$.
\end{enumerate}

To prove Theorem \ref{thm:Lyubeznik-graphs}, we establish that:

\begin{itemize}
    \item[(L1)] (forbidden structures) if $G$ is a 5-path $P_5$, 4-cycle $C_4$, 5-cycle $C_5$, 4-complete graph $K_4$, kite graph \kitesymb, gem graph \gemsymb, tadpole graph \tadpolesymb, butterfly graph \butterflysymb, or net graph \sunletsymb{3}, then $I(G)$ does not have a minimal Lyubeznik resolution (Proposition \ref{prop:non-Lyubeznik-graphs});
    \item[(L2)] (graph-theoretic classification) $G$ does not contain an induced subgraph of the forms listed in (L1) if and only if $G = L(a,b,c)$ for $a,b,c \in \ZZ_{\ge 0}$ (Proposition \ref{prop:LyubeznikClass}); and
    \item[(L3)] (Lyubeznik graphs) if $G = L(a,b,c)$, for $a,b,c \in \ZZ_{\ge 0}$, then $I(G)$ has a minimal Lyubeznik resolution (Proposition \ref{prop:Lyubeznik-graphs}).
    \end{itemize}

The proof of Theorem \ref{thm:bridge-friendly-chordal} follows in similar steps; particularly, we show that:

\begin{enumerate}
    \item[(BF1)] (forbidden structures) if $G$ is a 4-complete graph $K_4$ \Kfoursymb, gem graph \gemsymb, kite graph \kitesymb, or net graph \sunletsymb{3}, then $I(G)$ is not bridge-friendly (Proposition \ref{prop:non-bridgefriendly-graphs});
    \item[(BF2)] (graph-theoretic classification) a chordal graph $G$ is does not contain an induced subgraph of the forms listed in (BF1) if and only if $G = \text{BF}(T,\w)$, for a tree $T$ and an edge-weight function $\w$ on $T$ (Proposition \ref{prop:BF-characterization-forbidden}); and
    \item[(BF3)] (bridge-friendly graphs) if $G = \text{BF}(T,\w)$, for some tree $T$ and edge-weight function $\w$, then $G$ is chordal (Proposition \ref{prop:BF-graph-characterization}) and $I(G)$ is bridge-friendly (Proposition \ref{prop:bridefriendly-graphs}).
    \end{enumerate}

The study of higher powers $I(G)^n$, with $n \ge 2$, proceeds in a similar fashion, though considerably simpler. We prove in Theorems \ref{thm:Lyubeznik-powers} and \ref{thm:BF-powers} that:
\begin{enumerate}
    \item[(\textbf{\ref{thm:Lyubeznik-powers}})] $I(G)^n$ has a minimal Lyubeznik resolution if and only if either $G$ is an edge, or $n=2$ and $G$ is a path of length 2; and
    \item[(\textbf{\ref{thm:BF-powers}})] $I(G)^n$ is bridge-friendly if and only if either $G$ is an edge, or $G$ is a path of length 2, or $n=2,3$ and $G$ is the triangle $C_3$.
\end{enumerate} 

Theorems \ref{thm:Lyubeznik-powers} and \ref{thm:BF-powers} are proved in a similar manner to that of Theorems \ref{thm:Lyubeznik-graphs} and \ref{thm:bridge-friendly-chordal}, though simpler. Particularly, we show that:

\begin{enumerate}
    \item[(Lp)] (forbidden structures) if $G$ is $K_{1,3}$, a 4-path $P_4$, a triangle $C_3$, or a 4-cycle $C_4$, then $I(G)^2$ does not have a minimal Lyubeznik resolution; and for the last 3 graphs in the list, neither does $I(G)^n$ for all $n \ge 2$ (see Proposition \ref{prop:non-Lyubeznik-powers}); and
    \item[(BFp)] (forbidden structures) if $G$ is a 4-star $K_{1,3}$, a 4-path graph $P_4$, a 4-cycle graph $C_4$, paw graph \pawsymb, diamond graph \diamondsymb, or a 4-complete graph $K_4$ \Kfoursymb, then $I(G)^2$ and $I(G)^3$ are not bridge-friendly; and for the first three graphs in the list, $I(G)^n$ is not bridge-friendly for all $n \ge 2$ (see Proposition \ref{prop:non-bridgefriendly-powers}). 
\end{enumerate}

\section*{Acknowledgements}

The first and third authors were partially supported by the NSF grants DMS 1801285 and 2101671. The first author was also partially supported by the NSF grant DMS 2001368. The second author acknowledges supports from a Simons Foundation grant. The third author made extensive use of the computer algebra systems \texttt{Sage} \cite{sagemath} and \texttt{Macaulay2} \cite{M2}, and the package \texttt{nauty} \cite{nauty}; the use of these is gratefully acknowledged. 

\section{Preliminaries}\label{section2}

In this section, we collect basic terminology and notations about graphs and edge ideals of graphs. We also give auxiliary results on Lyubeznik and Barile-Macchia resolutions, bridge-friendly property, and HHZ-subideals.

\subsection{Graphs and edge ideals of graphs}
Throughout the paper, $G=(V,E)$ denotes a connected simple graph with vertex set $V$ and edge set $E$, where $\card{E} \ge 1$. Let $\Bbbk$ be a field and let $S = \Bbbk[V]$ be the polynomial ring, whose variables are identified with the vertices in $G$. The \emph{edge ideal} of $G$ is define to be
\[
I(G) = \left\langle x_ix_j ~\middle|~ \{x_i,x_j\} \in E \right\rangle \subseteq S.
\]

A graph $H$ is an \emph{induced subgraph} of $G$ if the vertices of $H$ are vertices of $G$, and the edges of $H$ are exactly the edges of $G$ that connect two vertices in $H$. A graph is called a \emph{tree} if it has no cycles. A \emph{chord} in a cycle is an edge connecting two nonconsecutive vertices in the cycle. The graph $G$ is called \emph{chordal} if every cycle of length $\ge 4$ has a chord.

For a vertex $x \in V$, we call the vertices $\left\{y \in V ~\middle|~ \{x,y\} \in E \right\}$ its \emph{neighbors}. The \emph{distance} between two vertices $x$ and $y$ of $G$, denoted by $\dist_G(x,y)$, is defined the smallest value $n$ such that there exists a path of length $n$ connecting $x$ and $y$ in $G$. In particular, the neighbors of a vertex are exactly those of distance 1 from the given vertex.

We shall denote by $K_{1, n-1}$, $P_n$, $C_n$, and $K_n$ the complete bipartite graph of size $(1,n-1)$, the path with $n$ vertices, the cycle on $n$ vertices, and the complete graph on $n$ vertices, respectively. We sometimes refer to those graphs as the \emph{$n$-star}, \emph{$n$-path}, \emph{$n$-cycle}, and \emph{$n$-complete} graphs. Note that the $n$-path $P_n$ is said to have length $n-1$.

We call the following small graphs by particular names that their shapes represent: \emph{net} (\sunletsymb{3}), \emph{kite} (\kitesymb), \emph{diamond} (\diamondsymb), \emph{paw} (\pawsymb), \emph{gem} (\gemsymb), \emph{butterfly} (\butterflysymb), and \emph{tadpole} (\tadpolesymb).

\begin{figure}[!htb]
\minipage{0.33\textwidth}%
  \begin{center}\sunletsymb[4]{3}\end{center}
  \caption{Net}
\endminipage\hfill
\minipage{0.33\textwidth}%
  \begin{center}\kitesymb[4]\end{center}
  \caption{Kite}
\endminipage\hfill
\minipage{0.33\textwidth}%
  \begin{center}\diamondsymb[4]\end{center}
  \caption{Diamond}
\endminipage

\vspace*{0.4cm}

\minipage{0.33\textwidth}
  \begin{center}\pawsymb[3]\end{center}
  \caption{Paw}
\endminipage\hfill
\minipage{0.33\textwidth}
  \begin{center}\gemsymb[3]\end{center}
  \caption{Gem}
\endminipage\hfill
\minipage{0.33\textwidth}%
  \begin{center}\butterflysymb[3]\end{center}
  \caption{Butterfly}
\endminipage

\vspace*{0.4cm}

\minipage{0.5\textwidth}%
    \begin{center}\tadpolesymb[3.5]\end{center}
  \caption{Tadpole}
\endminipage

\end{figure}

\subsection{Taylor resolutions}

Let $I \subseteq S$ be a homogeneous ideal. A \emph{free resolution} of $S/I$ is a complex of free $S$-modules of the form
\[
\mathcal{F}: 0\to F_p\xrightarrow{\partial_p} F_{p-1}\to \cdots \to F_1 \xrightarrow{\partial_1} F_0\to 0
\]
where $H_0(\mathcal{F})\cong S/I$ and $H_i(\mathcal{F})\cong 0$ if $i\neq 0$. Moreover, $\mathcal{F}$ is \emph{$\mathbb{N}^r$-graded} if $\partial_i$ is $\mathbb{N}^r$-homogeneous for all $i$, and \emph{minimal} if $\partial_i(F_i)\subseteq (x_1,\dots, x_r)F_{i-1}$ for all $i$.

For a monomial ideal $I \subseteq S$, let $\mingens(I)$ denote its unique set of minimal monomial generators. We consider the full $\card{\mingens(I)}$-simplex whose vertices are labelled by the monomial generators of $I$.  It is well-known (cf. \cite{Tay66}) that the chain complex of this simplex gives a free resolution of $S/I$, which is referred to as the \emph{Taylor resolution}. Set $\card{\mingens(I)}=c$. Then the Taylor resolution of $S/I$ is of the form
\[
0\to S^{\binom{c}{c}} \to S^{\binom{c}{c-1}} \to \cdots \to S^{\binom{c}{1}} \to S^{\binom{c}{0}} \to 0.
\]
We remark that for each integer $i$, one can identify a basis of $S^{\binom{c}{i}}$ with the collection of subsets of $\mingens(I)$ with exactly $i$ elements. 

\subsection{Lyubeznik and Barile-Macchia resolutions}\label{subsection2.1}

The Lyubeznik and Barile-Macchia resolutions, introduced in \cite{Ly88, BM20, BW02, CK24}, are subcomplexes of the Taylor resolution. While the Taylor resolution accounts for all subsets of $\mingens(I)$, the Lyubeznik and Barile-Macchia resolutions are constructed based on \emph{Lyubeznik-critical} and \emph{Barile-Macchia-critical} subsets of $\mingens(I)$. The terms \emph{*-critical} depend on the choice of a given total order $(\succ)$ on $\mingens(I)$. Particularly, Lyubeznik-critical and Barile-Macchia-critical subsets with respect to $(\succ)$ are characterized in Propositions \ref{prop:Lyubeznik-subsets-in-general} and \ref{prop:BM-subsets-in-general} below.

\begin{proposition}[\protect{\cite[Theorem 3.2]{BW02}}]\label{prop:Lyubeznik-subsets-in-general}
    Let $\sigma=\{m_1,\dots, m_k \}$ be a subset of $\mingens(I)$, where $m_1\succ \cdots \succ m_k$. Then, $\sigma$ is Lyubeznik-critical if and only if the set
    \[
    \{ m\in \mingens(I)\colon m\mid \lcm(\{m_1, \dots, m_t\}) \text{ for some } 1< t \leq k \text{ where } m_t \succ m \}
    \]
    is empty.
\end{proposition}

\begin{example}\label{ex:4cycle-BM}
    Consider the ideal $I=(xw,xy,yz,zw) \subseteq \Bbbk[x,y,z,w]$ with the total ordering $wx\succ xy\succ yz\succ zw$. 
    \begin{enumerate}
        \item[(a)] Let $\sigma_1=\{xw,xy,yz\}$. Since $zw\mid \lcm(\sigma_1)$ and $yz\succ zw$, the set $\sigma_1$ is not Lyubeznik-critical.
        \item[(b)] Let $\sigma_2=\{xw,xy,zw\}$. It is routine to verify that $\sigma_2$ is Lyubeznik-critical. Indeed, set $m_1=xw, m_2=xy$, and $m_3=yz$. Then by brute force, one can see that there exists no $m\in \mingens(I)$ such that either 
        \[
        m\mid \lcm(m_1,m_2) \text{ and } m_2\succ m
        \]
        or
        \[
        m\mid \lcm(m_1,m_2,m_3) \text{ and } m_3\succ m.
        \]
    \end{enumerate}
\end{example}

\begin{definition} \quad \quad \label{defi:types}
\begin{enumerate} 
    \item Given $\sigma\subseteq \mingens(I)$ and $m\in \mingens(I)$  such that $\lcm(\sigma \cup \{m\})=\lcm(\sigma\setminus \{m\})$. We say that $m$ is a \emph{bridge} (respectively, \emph{gap}) of $\sigma$ if $m\in \sigma$ (respectively, $m\notin \sigma$).
    \item If $m\succ m'$ where $m,m'\in \mingens(I)$, we say that $m$ \emph{dominates} $m'$.
    \item The \emph{smallest bridge function} is defined to be
    \[
    \sbridge: \mathcal{P}(\mingens(I))\to \mingens(I) \sqcup \{\emptyset\}
    \]
    where $\mathcal{P}(\mingens(I))$ denotes the power set of $\mingens(I)$, and $\sbridge(\sigma)$ is the smallest bridge of $\sigma$ (with respect to $(\succ)$) if $\sigma$ has a bridge and $\emptyset$ otherwise.
    \item A monomial $m\in \mingens(I)$ is called a \emph{true gap} of $\sigma\subseteq \mingens(I)$ if 
        \begin{enumerate}
            \item[(a)]  it is a gap of $\sigma$, and 
            \item[(b)]  the set $\sigma \cup \{m\}$ has no new bridges dominated by $m$. In other words, if $m'$ is a bridge of $\sigma \cup \{m\}$ and $m\succ m'$, then $m'$ is a bridge of $\sigma$.
        \end{enumerate}
    Equivalently, $m$ is not a true gap of $\sigma$ either if $m$ is not a gap of $\sigma$ or if there exists $m'\prec m$ such that $m'$ is a bridge of $\sigma \cup \{m\}$ but not one of $\sigma$. In the latter case, we call $m'$ a \emph{non-true-gap witness} of $m$ in $\sigma$.
    \item A subset $\sigma\subseteq \mingens(I)$ is called \emph{potentially-type-2} if it has a bridge not dominating any of its true gaps, and \emph{type-1} if it has a true gap not dominating any of its bridges. Moreover, $\sigma$ is called \emph{type-2} if it is potentially-type-2 and whenever there exists another potentially-type-2 $\sigma'$ such that 
    \begin{equation*} 
    	\sigma' \setminus \{\sbridge(\sigma')\}=\sigma \setminus \{\sbridge(\sigma)\},
    \end{equation*}
    we have $\sbridge(\sigma')\succ \sbridge(\sigma)$.
\end{enumerate}
\end{definition}

We provide an explicit example of these concepts.

\begin{example}\label{ex:4cycle-BM-2}
    Consider the ideal $I=(xw,xy,yz,zw)$ with the total ordering $wx\succ xy\succ yz\succ zw$. 
    \begin{enumerate}
        \item[(a)] Let $\sigma_1=\{xw,xy,yz\}$. It is clear that $zw$ is the only true gap and $xy$ is the only bridge of $\sigma_1$, and by definition, $\sigma_1$ is type-1.
        \item[(b)] Let $\sigma_2=\{xw,xy,zw\}$. It is clear that $yz$ is the only gap and $xw$ is the only bridge of $\sigma_2$. However, $yz$ is not a true gap $\sigma_2$ and by definition, $\sigma_2$ is potentially-type-2. Moreover, $\sigma_2$ is not type-2 since for $\sigma_2'=\{xy,yz,zw\}$ (which one can check to be potentially-type-2), we have 
        \[
        \sigma_2' \setminus \{\sbridge(\sigma_2')\}=\sigma_2 \setminus \{\sbridge(\sigma_2)\},
        \]
        and $\sbridge(\sigma_2)=xy\succ yz =\sbridge(\sigma_2')$.
        \item[(c)] Let $\sigma_3=\{xw,xy,yz,zw\}$. All elements in $\sigma_3$ are its bridges. Hence $\sigma_3$ is potentially-type-2, and by definition, it is easy to see that it is indeed type-2.
    \end{enumerate}
\end{example}

\begin{proposition}[\protect{\cite[Theorem 2.24]{CK24}}]\label{prop:BM-subsets-in-general}
    Let $\sigma$ be a subset of $\mingens(I)$. Then, $\sigma$ is Barile-Macchia-critical if and only if it is neither type-1 nor type-2.
\end{proposition}

The following theorem, resulting from applications of discrete Morse theory, describes Lyubeznik and Barile-Macchia resolutions with respect to a given total order $(\succ)$ on $\mingens(I)$.

\begin{theorem}[\protect{\cite[Propositions 2.2 and 3.1]{BW02}}]\label{thm:Morseres}
    Let $I$ be a monomial ideal with a fixed total order $(\succ)$ on $\mingens(I)$. Let $\mathcal{F}$ be the Lyubeznik (respectively, Barile-Macchia) resolution  of $S/I$ with respect to $(\succ)$. Then, for any integer $i$, a basis of $\mathcal{F}_i$ can be identified with the collection of Lyubeznik-critical (respectively, Barile-Macchia-critical) subsets of $\mingens(I)$ with exactly $i$ elements.
\end{theorem}

\subsection{Bridge-friendly monomial ideals}

The terminology ``bridge-friendly monomial ideals'' was introduced in \cite{CK24} to ease the process of identifying Barile-Macchia-critical subsets of the generators. The motivation for this concept is partly that potentially-type-2 subsets are easier to check than type-2 subsets.

\begin{definition} \label{def:bridgefriendly}
    A monomial ideal $I$ is called \emph{bridge-friendly} (with respect to $(\succ)$) if all potentially-type-2 subsets of $\mingens(I)$ are type-2. Equivalently, $I$ is bridge-friendly if and only if Barile-Macchia-critical subsets of $\mingens(I)$ are precisely the ones that have neither bridges nor true gaps.
\end{definition}

It turns out that most monomial ideals are bridge-friendly (with respect to a total order of the generators) --- see \cite[Theorem 5.4]{CK24} --- and the Barile-Macchia resolutions of bridge-friendly ideals are minimal --- see \cite[Theorem 2.29]{CK24}. 
We shall identify the failure of being bridge-friendly. 

For the remainder of the section, let $\sigma$ denote a non-empty subset of $\mingens(I)$.

\begin{proposition}[\protect{\cite[Proposition 2.21]{CK24}}]\label{prop:true-gap-not-dominate-bridges}
    A monomial $m$ is a gap of $\sigma$ such that  $\sbridge(\sigma \cup \{m\})=m$ if and only if $m$ is a true gap of $\sigma$ that does not dominate any bridge of $\sigma$.
\end{proposition}

\begin{lemma}\label{lem:existence-m1-m2-m3}
   A monomial ideal $I$ is not bridge-friendly with respect to $(\succ)$ if and only if there exist a type-1 set $\tau \subseteq \mingens(I)$ and monomials $m_1\succ m_2$ in $\mingens(I)$ such that:
   \begin{enumerate}
       \item The monomials $m_1$ and $m_2$ are true gaps of $\tau$ that do not dominate any bridges (of $\tau$). In particular, $m_1,m_2 \notin \tau$.
       \item The sets $\tau\cup \{m_1\}$ and  $\tau\cup \{m_2\}$ are potentially-type-2.
   \end{enumerate}
   In this case, $m_2$ can be chosen to be the smallest true gap of $\tau$. Moreover, under these conditions, there exists a monomial $m_3\prec m_2$ such that $m_3$ is a bridge of $\tau \cup \{m_1,m_2\}$. In particular, $m_3\in \tau$.
\end{lemma}

\begin{proof}
    The ideal $I$ is not bridge-friendly if and only if there exists a set $\sigma\in \mingens(I)$ that is potentially-type-2, but not type-2. By definition, this is equivalent to saying that there exists a different potentially-type-2 set $\sigma'\subseteq \mingens(I)$ such that
    \[
    \sigma' \setminus \{\sbridge(\sigma')\}=\sigma \setminus \{\sbridge(\sigma)\},
    \]
    and $\sbridge(\sigma)\succ \sbridge(\sigma')$. Set 
        $m_1 = \sbridge(\sigma),\ m_2 = \sbridge(\sigma'),\text{ and }
        \tau = \sigma \setminus \{m_1\}= \sigma' \setminus \{m_2\}.$ 
    Then, using Proposition \ref{prop:true-gap-not-dominate-bridges}, one can see that this condition is equivalent to $(1)$ and $(2)$, as claimed. 
    
    We remark that $m_2$ is a true gap of $\tau$ by definition. We can choose $m_2$ to be the smallest true gap of $\tau$ since then $m_2=\sbridge(\tau \cup \{m_2\})$ and $\tau\cup \{m_2\}$ is potentially-type-2 by \cite[Remark 2.26]{CK24}. Thus, the previous part can still proceed in this setting.
    Moreover, since $\sigma$ is potentially-type-2, $m_2$ is not a true gap of $\sigma$ by definition, and so $\tau \cup \{m_1,m_2\}=\sigma \cup \{m_2\}$ has a bridge $m_3\prec m_2$, as claimed.
\end{proof}

The equivalent condition to non-bridge-friendliness in fact says a lot more about the three monomials $m_1\succ m_2 \succ m_3$. We will first introduce a new terminology.

\begin{definition}
    Let $I$ be a monomial ideal, $x$ a variable, $\sigma \subset \mingens(I)$, and $m, m' \in \sigma$. We say that two monomials $m$, $m'$ \emph{share a factor $x^n$ unique within $\sigma$} if $n\geq 1$ and we have the following:
    \begin{enumerate}[label=(\roman*)]
        \item $x^n\mid m,m'$ but $x^{n+1}\nmid m,m'$.
        \item $x^n\nmid m''$ for all $m''\in \sigma \setminus \{m,m'\}$.
    \end{enumerate}
\end{definition}

This concept appears many times when we work with gaps that are not true gaps, as shown in the next result.

\begin{proposition}\label{prop:gap-but-not-true-gap}
    If a monomial $m\in \mingens(I)$ is a gap, but not a true gap of $\sigma$, then $m$ and any of its non-true-gap witnesses $m'$ share a factor unique within $\sigma$.
\end{proposition}
   
\begin{proof}
     Since $m'$ is not a bridge of $\sigma$, there exists a factor $x^n$ such that $x^n\mid m'$, $x^{n+1}\nmid m'$, and $x^n\nmid m''$ for each $m''\in \sigma\setminus\{m'\}$. Moreover, since $m'$ is a bridge of $\sigma\cup \{m\}$, we have 
    \[
    x^n \mid m'\mid \lcm(\sigma \cup \{m\}),
    \]
    and thus $x^n \mid m$. On the other hand, we have $x^{n+1}\nmid m$ since $\lcm(\sigma \cup \{m\})=\lcm (\sigma)$.    
\end{proof}

\begin{remark}\label{rem:unique-factor-m1-m2-m3}
    In the proof of Lemma \ref{lem:existence-m1-m2-m3}, we showed that $m_2$ is a gap of $\tau \cup \{m_1\}$ and set $m_3$ to be a non-true-gap witness of $m_2$ in $\tau \cup \{m_1\}$. By Proposition \ref{prop:gap-but-not-true-gap}, $m_2$ and $m_3$ share a factor unique within $\tau \cup \{m_1\}$. It can be checked that $m_3$ is also a non-true-gap witness of $m_1$ in $\tau \cup \{m_2\}$. Therefore, $m_1$ and $m_3$ share a factor unique within $\tau \cup \{m_2\}$.
\end{remark}

This observation immediately gives a condition to determine true gaps:

\begin{corollary}\label{cor:true-gap}
    If a monomial $m$ is a gap of $\sigma$ such that 
    \[
    m\mid\lcm(\{m'\in \sigma \colon  m'\succ m \text{ or $m$ and $m'$ do not share a factor unique within $\sigma$}  \}),
    \]
    then $m$ is a true gap of $\sigma$.
\end{corollary}

\subsection{Restriction lemmas and HHZ-subideals}

There is a one-to-one correspondence between $\mathbb{N}^r$ and the set of monomials in $S$. Thus, at times we will abuse notations and use monomials (in $S$) and vectors (in $\mathbb{N}^r$) interchangeably.

Fix a monomial ideal $I \subseteq S$ and a monomial $m \in S$. Let $I^{\leq m}$ be the monomial ideal generated by elements of $\mingens(I)$ that divides $m$. It is clear that $I^{\leq m}$ is always a subideal of $I$. This notation was introduced in \cite{HHZ04}, although the idea briefly appeared prior in \cite{BPS98}. We will call $I^{\leq m}$ a \emph{Herzog-Hibi-Zheng-subideal} (with respect to $m$) of $I$, or \emph{HHZ-subideal} for short. 

Let $\mathcal{F}: 0\to F_p \to F_{p-1}\to \dots \to F_1\to F_0 \to 0$
be a (minimal) $\mathbb{N}^r$-graded free resolutions of $S/I$, with $F_i= \bigoplus_{j \in \ZZ} R(-q_{ij})$. Let $\mathcal{F}^{\leq m}$ be the subcomplex of $\mathcal{F}$, whose $i$-th syzygy module is $\bigoplus_{q_{ij}\leq m} R(-q_{ij})$. Here, for monomials $a$ and $b$, we write $a\leq b$ to mean $a\mid b$. 

By \cite[Lemma 4.4]{HHZ04}, $\mathcal{F}^{\leq m}$ is a (minimal) $\mathbb{N}^r$-graded free resolution of $S/I^{\leq m}$. This result is referred to as the \emph{Restriction Lemma}. We will also state analogs of the Restriction Lemma for Lyubeznik, Barile-Macchia resolutions, and the bridge-friendly property. These analogs come from the observation that, for a fixed a monomial $m$, faces $\sigma$ in the simplex over the vertices $\mingens(I)$, whose least common multiple of the vertices divides $m$, are exactly the faces of the simplex over the vertices $\mingens(I^{\leq m})$.

\begin{lemma}\label{lem:HHZ-BM-Lyubeznik}
    (Restriction Lemma for Barile-Macchia/Lyubeznik resolutions) Let $I$ be a monomial ideal and let $m$ be a monomial in $S$. If $\mathcal{F}$ is a (minimal) Barile-Macchia (respectively, Lyubeznik) resolution of $S/I$, then $\mathcal{F}^{\leq m}$ is a (minimal) Barile-Macchia (respectively, Lyubeznik) resolution of $S/I^{\leq m}$.
\end{lemma}

\begin{lemma}\label{lem:HHZ-bridge-friendly}
    (Restriction Lemma for bridge-friendly ideals) Let $I$ be a monomial ideal and let $m$ be a monomial in $S$. If $I$ is bridge-friendly with respect to a total order $(\succ)$, then so is $I^{\leq m}$ with respect to the total order induced from $(\succ)$.
\end{lemma}

As an application of Lemmas \ref{lem:HHZ-BM-Lyubeznik} and \ref{lem:HHZ-bridge-friendly}, consider a graph $G$ and an induced subgraph $H$. Let $m$ be the product of all the vertices in $H$. Then, $I(H)^n=(I(G)^n)^{\leq m^n}$ is an HHZ-subideal of $I(G)^n$, for each integer $n$. Thus, we arrive at the following corollary. 

\begin{corollary}\label{cor:induced-subgraph-Morse}
    Let $G$ be a graph and let $H$ be an induced subgraph of $G$. For any integer $n$, if $I(G)^n$ is Lyubeznik, Barile-Macchia, or bridge-friendly, then so is $I(H)^n$.
\end{corollary}

The following result is obvious from definition, which allows us to discuss about the property of being Lyubeznik, Barile-Macchia, and bridge-friendly for a monomial ideal $I$ and $mI$ interchangeably, where $m$ is any monomial.

\begin{proposition} \label{prop:mI}
    Let $I$ be a monomial ideal and let $m$ be a monomial. Then $I$ is Lyubeznik, Barile-Macchia, or bridge-friendly, if and only if so is $mI$.
\end{proposition}

An affirmative answer to the following question would allow us to deduce the property of being Lyubeznik, Barile-Macchia, and bridge-friendly of a lower power of a monomial ideal from its higher powers.

\begin{proposition}\label{prop:HHZ-subideal-edgeideal}
    Let $G$ be the $P_4$, $C_3$, or $C_4$ graph. Then for each positive integer $n$, there exists a generator $f\in \mingens(I(G))$ such that the ideal $fI(G)^{n}$ is an HHZ-subideal of $I(G)^{n+1}$.
\end{proposition}
\begin{proof}
    Fix an integer $n$. We will prove the statement in each case.
    \begin{enumerate}
        \item $G = P_4$, i.e.,  $I\coloneqq I(G)=(x_1x_2,x_2x_3,x_3x_4)$. \newline
        Set $J \coloneqq (x_1x_2,x_2x_3)$, $m \coloneqq x_2^n(x_1x_3x_4)^{n+1}$, and $f \coloneqq x_3x_4$. We have
        \[
        I^{n+1}=J^{n+1}+fI^n.
        \]
        In fact, in this case, it is clear that
        \[
        \mingens(I^{n+1})=\mingens(J^{n+1}) \sqcup \mingens(fI^n).
        \]
        It is straightforward that all elements of $\mingens(I^{n+1})$ that divide $m$ come from $\mingens(fI^n)$. Therefore,  the ideal $fI^n$ is the HHZ-subideal of $I^{n+1}$ with respect to $m$, as desired.
        \item Similar arguments apply to the case where $G = C_3$, i.e., $I(G)=(x_1x_2,x_2x_3,x_1x_3)$.
        \item $G = C_4$, i.e., $I\coloneqq I(G)=(x_1x_2,x_2x_3,x_3x_4,x_1x_4)$. \newline
        Set $J \coloneqq (x_1x_2,x_2x_3,x_3x_4)$, $m \coloneqq (x_2x_3)^n(x_1x_4)^{n+1}$, and $f \coloneqq x_1x_4$. We have
        \[
        I^{n+1}=J^{n+1}+fI^n.
        \]
        We claim that the HHZ-subideal of $I^{n+1}$ with respect to $m$ is exactly $fI^n$. Indeed, consider any element $g$ in $\mingens(I^{n+1})$. If $g$ belongs to $fI^n$, then it clearly divides $m$. On the other hand, if it does not, i.e., $g$ is not divisible by $f=x_1x_4$, then $g$ must be of the form
        \[
        (x_1x_2)^a(x_2x_3)^{n+1-a} \text{ or } (x_2x_3)^a(x_3x_4)^{n+1-a}
        \]
        for some integer $a$. In either case, $g$ does not divide $m$. Thus the claim holds, and therefore $fI^n$ is an HHZ-subideal of $I^{n+1}$.\qedhere
    \end{enumerate}
\end{proof}

We remark that Proposition \ref{prop:HHZ-subideal-edgeideal} does not necessarily hold for monomial ideals or edge ideals in general. Consider $I=(xx_1,xx_2,xx_3)$. One can check that for any monomial $m$, the ideal $(I^3)^{\leq m}$ does not have the same total Betti numbers as $fI^2$ for any generator $f\in \mingens(I)$. This implies that $fI^2$ is not an HHZ-subideal of $I^3$. As a consequence, $fI(K_{1,3})^2$ is not an HHZ-subideal of $I(K_{1,3})^3$ for any generator $f\in \mingens(I(K_{1,3}))$.

\section{Powers of edge ideals with minimal Lyubeznik resolutions}

In this section, we classify all graphs $G$ and integers $n$ such that $I(G)^n$ has a minimal Lyubeznik resolution. We will begin with the case when $n=1$, which is considerably more complicated, and follow with the case when $n \ge 2$.

We shall first identify the forbidden structures for Lyubeznik graphs. The following statement is verified by \texttt{Macaulay2} \cite{M2} computations. We remark here that, to verify whether a Lyubeznik resolution is minimal, it suffices to check in characteristic 2. This is because the coefficients in the differentials of Lyubeznik resolutions are either 0 or a product of $\pm 1$ with a monomial. Thus all the \texttt{Macaulay2} verifications are done over the base field $\mathbb{Z}/2\mathbb{Z}$.  Lyubeznik-critical subsets of $\mingens(I)$ can be found in finite time just by definitions. We check if a monomial ideal is Lyubeznik by exhausting all of its Lyubeznik resolutions.

\begin{proposition}\label{prop:non-Lyubeznik-graphs}
    Let $G$ be one of the following graphs: \newline
    \begin{minipage}[t]{0.4\textwidth}
    \begin{enumerate}
        \item The $5$-path graph $P_5$.
        \item The $4$-cycle graph $C_4$.
        \item The $5$-cycle graph $C_5$.
        \item The complete graph $K_4$ \Kfoursymb. 
        \item The kite graph \kitesymb.
    \end{enumerate}
    \end{minipage}%
    \begin{minipage}[t]{0.5\textwidth}
    \begin{enumerate} 
    \setcounter{enumi}{5}
        \item The gem graph \gemsymb.
        \item The tadpole graph \tadpolesymb.
        \item The butterfly graph \butterflysymb.
        \item The net graph \sunletsymb{3}.
    \end{enumerate}
    \end{minipage}
    
    \noindent Then $I(G)$ is not Lyubeznik.
\end{proposition}
\begin{proof}
    Verified with \texttt{Macaulay2} computations over the base field $\ZZ/2\ZZ$.
\end{proof}

Proposition \ref{prop:non-Lyubeznik-graphs}, coupled with Corollary \ref{cor:induced-subgraph-Morse}, shows that any graph, whose edge ideal is Lyubeznik, cannot contain the graphs listed in Proposition \ref{prop:non-Lyubeznik-graphs} as induced subgraphs. 
The following construction plays a key role in the classification of Lyubeznik graphs.

\begin{definition} \label{def:Labc}
    Let $a,b,c$ be nonnegative integers. We define $L(a,b,c)$ to be the graph whose vertex and edge sets are:
    \[
    V(L(a,b,c)) = \{x,y\} \sqcup \{x_i\}_{i = 1}^{a} \sqcup \{y_j\}_{j = 1}^{b} \sqcup \{z_k\}_{k = 1}^{c}, 
    \]
    \[
     E(L(a,b,c)) = \{xy\} \sqcup \{xx_i\}_{i = 1}^{a} \sqcup \{yy_j\}_{j = 1}^{b} \sqcup \{xz_k\}_{k = 1}^{c} \sqcup \{yz_k\}_{k = 1}^{c}.
    \]
\end{definition}

Roughly speaking, $L(a,b,c)$ consists of exactly $c$ triangles $xyz_k$, for $k = 1, \dots, c$, that share a common edge $\{x,y\}$, together with $a$ leaves $\{x,x_i\}$, for $i = 1, \dots, a$, attached to $x$, and $b$ leaves $\{y,y_j\}$, for $j = 1, \dots, b$, attached to $y$. In particular, $L(a, b, c)$ has $a+b+c+2$ vertices.

\begin{example} \quad \quad
    \begin{enumerate}
        \setcounter{enumi}{-1}
        \item By convention, $L(0, 0, 0)$ is the graph consisting of exactly one edge.
        \item The graph $L(0,0,c)$ is exactly $c$ triangles sharing one edge. In particular, $L(0,0,1)$ is the triangle $C_3$, and $L(0,0,2)$ is a diamond \diamondsymb.
        \item The graph $L(a,0,0)$ is the star graph $K_{1, a+1}$.
        \item The graph $L(a,b,0)$ is a gap-free tree. In fact, $\left\{L(a,b,0) ~\middle|~ a,b \in \mathbb{Z}_{\ge 0}\right\}$ are exactly the connected graphs that admit minimal Scarf resolutions \cite[Theorem 8.3]{FHHM24}.
    \end{enumerate}
\end{example}

The forbidden structures in Proposition \ref{prop:non-Lyubeznik-graphs} can now be characterized using graphs of the form $L(a,b,c)$.

\begin{proposition} \label{prop:LyubeznikClass}
    A connected graph is $L(a,b,c)$ for some integers $a,b,$ and $c$ if and only if it does not contain, as an induced subgraph, any of the graphs listed in Proposition \ref{prop:non-Lyubeznik-graphs}.
\end{proposition}

\begin{proof}
    It is clear that for any $a,b,$ and $c$, the graph $L(a,b,c)$ does not contain any of the graphs listed in Proposition \ref{prop:non-Lyubeznik-graphs} as an induced subgraph. We will hence only show the other direction. Let $G$ be a connected graph that does not contain any of the forbidden graphs. We remark that since $G$ does not contain $P_5$, $C_4$, or $C_5$, it does not contain $P_n$ for any $n\geq 5$ or $C_n$ for any $n\geq 4$, either. In particular, $G$ is chordal. Moreover, since $G$ does not contain the $K_4$ graph, it does not contain the $K_n$ graph for any $n\geq 4$.
    
    Suppose $G$ contains a cycle with a unique chord as an induced subgraph. By Proposition \ref{prop:non-Lyubeznik-graphs} (2), that cycle must be a diamond. So we can assume that $G$ contains the diamond graph \diamondsymb \ formed by $w,x,y,z$ with a chord $wy$. We will show that $G$ is exactly $L(a,b,c)$ for some integers $a,b,c$. Indeed, consider any vertex $u$ of $G$ that does not belong to this diamond. It suffices to show that $N(u)$, the set of neighbors of $u$ among $a,b,c,d$, is $\{w\}$, $\{y\}$ or $\{w,y\}$. Suppose otherwise. Then by symmetry we can assume that $x\in N(u)$. we will derive a contradiction after considering all possibilities.
    \begin{itemize}
        \item $N(u)=\{x\}$: Then $G$ contains the kite graph \kitesymb, a contradiction.
        \item $N(u)=\{x,w\}$ or $N(u)=\{x,y\}$: Then $G$ contains the gem graph \gemsymb\  formed by $w,x,y,z,u$, a contradiction.
        \item $N(u)=\{x,z\}$: Then $G$ contains the $C_4$ graph formed by $w,x,u,z$, a contradiction.
        \item $N(u)=\{x,w,y\}$: Then $G$ contains the $K_4$ graph formed by $w,x,y,u$, a contradiction.
        \item $N(u)=\{x,w,z\}$ (or $N(u)=\{x,y,z\}$, resp): Then $G$ contains the $C_4$ graph formed by $x,y,z,u$ (or $w,x,y,u$, resp), a contradiction.
        \item $N(u)= \{x,w,y,z\}:$ Then $G$ contains the $K_4$ graph formed by $w,x,y,u$, a contradiction.
    \end{itemize}
    
    We shall consider the case that $G$ does not contain a cycle with a unique chord. If $G$ is a tree, it is straightforward from Proposition \ref{prop:non-Lyubeznik-graphs} (1) that $G$ is $L(a,b,0)$ for some integers $a,b$. If $G$ is not a tree, then by Proposition \ref{prop:non-Lyubeznik-graphs} (2), it contains a triangle, whose vertices we shall denote by $x,y,z$. By \cite[Theorem 2.2]{NK08}, either $G$ is a complete graph, or one vertex of the maximal complete induced subgraph that contains this triangle is a 1-cutset of $G$.  Since $G$ cannot contain any $K_n$ graph for any $n\geq 4$, the triangle formed by $a,b,c$ must be the maximal complete induced subgraph of $G$ that contains itself. Thus, we may assume that $z$ is a 1-cutset of $G$. By definition, this means that the graph $G\setminus \{z\}$ has at least two connected components. Let $V$ denote the set of vertices in the connected component that contains $x$ and $y$, and $W$ the set of the remaining vertices. 

    Since $G$ does not contain the tadpole graph \tadpolesymb, no vertex $w\in W$ satisfies $\dist_G(w,z)\geq 2$. In other words, all vertices in $W$ are connected to $z$ in $G$. On the other hand, since $G$ does not contain the butterfly graph, no two vertices in $W$ are connected. Therefore, the subgraph induced by $x,y,z$ and vertices in $W$ is:

    \begin{center}
    	\cricketsymb
    \end{center}
    
    Now, we consider vertices in $V$. Since $G$ does not contain any $n$-path graph $P_n$ where $n\geq 5$, no vertex $v\in V$ satisfies $\dist_G(v,x)\geq 2$ or $\dist_G(x,y)\geq 2$. In other words, all vertices in $V$ are connected to either $x$ or $y$ in $G$. Since $G$ does not the kite graph \kitesymb, no vertex in $V$ is connected to both $x$ and $y$ in $G$. Suppose there exist two vertices $x',y'\in X$ such that $x'$ is connected to $x$, not $y$, and $y'$ is connected to $y$, not $x$. Then $G$ either contains a $C_4$ graph formed by $x',y',x,y$ (if $x'$ and $y'$ are connected), or the net graph \sunletsymb{3} formed by $x,y,z,x',y',z'$ where $z'$ is a fixed vertex in $W$, a contradiction either way. Thus all vertices in $V$ are either connected to $x$ and not $y$, or vice versa. Thus $G=L(a,b,1)$ where $a=\card{V}-2$ and $b=\card{W}$.
\end{proof}

Before proceeding to showing that the edge ideals of $L(a,b,c)$ graphs are Lyubeznik, we make the following observation.

\begin{remark} Let $\mathcal{F}$ be a Lyubeznik resolution of $S/I$, where $I \subseteq S$ is a monomial ideal. Then, by Theorem \ref{thm:Morseres}, for any integer $i$, a basis of $\mathcal{F}_i$ can be identified with the collection of Lyubeznik-critical subsets of Mingens(I)
with exactly $i$ elements. Moreover, the differentials are as follows:
\[
\partial(\sigma) = \sum_{\tau\subseteq \sigma, \card{\tau}=\card{\sigma}-1} \pm 1 \frac{\lcm(\sigma)}{\lcm(\tau)} \tau,
\]
where $\sigma$ is a Lyubeznik-critical subset. From this description, it is easy to see that a Lyubeznik resolution is minimal if and only if each Lyubeznik-critical subset does not have any bridge. We shall make use of this observation often in this section. 
\end{remark}

\begin{proposition}\label{prop:Lyubeznik-graphs}
    If $G = L(a,b,c)$, for some nonnegative integers $a,b,$ and $c$, then the edge ideal $I(G)$ is Lyubeznik.
\end{proposition}

\begin{proof}
    If $G=L(a,0,0)=K_{1,a+1}$, then it is clear that $I(G)$ is Taylor, and thus any Lyubeznik resolution of $I(G)$ is minimal. For the rest of the proof, we assume that $G=L(a,b,c)$ where $a\geq 1$ or $c\geq 1$. Consider the following total order on $\mingens(I(G))$:
    \[
    yy_b \succ \cdots \succ yy_1 \succ xx_a \succ \cdots \succ xx_1 \succ yz_c \succ \cdots \succ yz_1 \succ xz_c \succ \cdots \succ xz_1 \succ xy.
    \]
    We will show that the Lyubeznik resolution of $I(G)$ induced from this total order is minimal. By definition, it suffices to show that any Lyubeznik-critical subset $\sigma$ of $\mingens(I(G))$ has no bridge.  
    
    Indeed, since $\sigma$ is Lyubeznik-critical, we have
    \[
    xy\nmid \sigma\cap \{e\in E(G)\colon e\succ xy\}.
    \]
    In particular, if an edge incident to $x$ is in $\sigma$, then $\sigma$ does not contain any edge incident to $y$, and vice versa. Suppose that $\sigma$ contains a leaf edge which, without loss of generality, we can assume to be $xx_1$. Then $\sigma$ is a subset of
    \[
    \{xy\} \sqcup \{xx_i\}_{i=1}^a \sqcup \{xz_k\}_{k=1}^c,
    \]
    and thus has no bridge. On the other hand, suppose that $\sigma$ contains no leaf edges. If $\sigma=\{xy\}$, then it clearly has no bridge. Otherwise, without loss of generality, we can assume that $\sigma$ contains $xz_1$. By the same observation, $\sigma$ is then a subset of 
    \[
    \{xy\} \sqcup \{xz_k\}_{k=1}^c,
    \]
    and thus has no bridge. This concludes the proof.
\end{proof}

Combining the preceding three propositions, we obtain the first main result of the paper.

\begin{theorem}\label{thm:Lyubeznik-graphs}
    The edge ideal $I(G)$ is Lyubeznik if and only if $G$ is  $L(a,b,c)$ for some nonnegative integers $a,b,c$.
\end{theorem}

\begin{remark} \label{rm:disconnected} 
Theorem \ref{thm:Lyubeznik-graphs} can be naturally extended to work over disconnected graphs. This is because the resolution of the edge ideal of a disconnected graph can be obtained by tensoring the corresponding resolutions of its connected components.
\end{remark}

For the remaining of this section, we will investigate Lyubeznik resolutions of higher powers of edge ideals of graphs.

\begin{proposition}\label{prop:non-Lyubeznik-powers}
    Let $G$ be the $K_{1,3}$, $P_4$, $C_3$, or $C_4$ graph.   Then, the ideal $I(G)^2$ is not Lyubeznik. Moreover, if $G$ is among the last three graphs of the given list, then $I(G)^n$ is not Lyubeznik for any $n\geq 2$.
\end{proposition}

\begin{proof}
    For the first three graphs in the list, $I(G)^2$ is generated by 6 monomials, so there are $6!$ possibilities. The first assertion for these graphs can be verified by Macaulay2 computations. For $C_4$, one can verify with Macaulay2 that the ideal $(x_1^2x_2^2,\ x_2^2x_3^2,\ x_1^2x_2x_4,\ x_1x_2^2x_3,\ x_2x_3^2x_4,\ x_1x_2x_3x_4)$, an HHZ-subideal of $(x_1x_2,\ x_2x_3,\ x_3x_4,\ x_1x_4)^2$ with respect to $(x_1x_2x_3)^2x_4$, does not admit a minimal Lyubeznik resolution. Thus, the first assertion for $C_4$ follows from Lemma \ref{lem:HHZ-BM-Lyubeznik}.

    The second statement is a direct consequence of the first assertion, Proposition \ref{prop:HHZ-subideal-edgeideal}, and Lemma \ref{lem:HHZ-BM-Lyubeznik}.
\end{proof}

We are now ready to state the next main result of this section.

\begin{theorem}\label{thm:Lyubeznik-powers}
    Let $G$ be a graph and let $n \ge 2$ be a positive integer. Then the ideal $I(G)^n$ is Lyubeznik if and only if one of the following holds:
    \begin{enumerate}
        \item $G$ is an edge; or 
        \item $n=2$ and $G$ is a path of length 2.
    \end{enumerate}
\end{theorem}

\begin{proof}
The ``if" implication is straightforward. We will prove the ``only if" direction.
    
    Assume that $G$ is a graph such that $I(G)^n$ is Lyubeznik. By Proposition \ref{prop:non-Lyubeznik-powers}, $G$ does not contain, as an induced subgraph, any $C_k$, where $k\geq 3$, or $P_k$, where $k\geq 4$. In other words, $G$ must be a star graph $K_{1,k}$ for some integer $k$. 
    
    If $n=2$, then $G$ does not contain $K_{1,3}$ as an induced subgraph by Proposition \ref{prop:non-Lyubeznik-powers} (1). Thus $k\leq 2$, and $G$ is either an edge or a path of length 2. 
    
    Consider the case that $n>2$. We will show that $I(K_{1,2})^n$ is not Lyubeznik. This would imply that $G$ does not contain $K_{1,2}$ as an induced subgraph and, hence, the desired statement would follow. To this end, we will exhibit that any Lyubeznik resolution of $S/I(K_{1,2})^n$ has length at least $3$ and, thus, is not minimal, due to the Auslander-Buchsbaum-Serre theorem. 
    
    Since the Lyubeznik resolutions of $I(K_{1,2})$ and its powers are the same as those of the ideal generated by two variables, it suffices to consider the ideal $(x,y)$. Fix any total order $(\succ)$ on $\mingens((x,y)^n)$ and let $m_1\prec m_2\prec m_3$ be the three smallest monomials with respect to $(\succ)$. We can set write
   	\begin{equation*} 
   	m_1 = x^ay^{n-a} \text{ and } m_2 = x^by^{n-b},
   	\end{equation*}
    for some integers $a$ and $b$. By symmetry, we can assume that $a>b$. 
    
    It suffices to construct a Lyubeznik-critical subset $\sigma$ of $\mingens((x,y)^n)$ of cardinality 3. Consider the sets
    \begin{equation*} 
    	\sigma_m \coloneqq \{m_1,m_2,m\} \text{ and } \tau_{m'} \coloneqq \{m_1,m_3,m'\},
    \end{equation*}
    where $m,m'\in \mingens((x,y)^n)$, $m\succ m_2$, and $m'\succ m_3$. Note that $m'$ always exists since $n>2$. By definition, $\sigma_m$ is Lyubeznik-critical if and only if $m_1\nmid \lcm(m_2,m)$, and $\tau_{m'}$ is Lyubeznik-critical if and only if $m_1\nmid \lcm(m_3,m')$ and $m_2\nmid \lcm(m_3,m')$. 
    
    The proof completes by showing that there always exists $m$ or $m'$ so that either $\sigma_m$ or $\tau_{m'}$ is critical. Indeed, if $a>1$, then $\sigma_m$ is Lyubeznik-critical for $m=x^{a-1}y^{n-a+1}$. On the other hand, if $a=1$, i.e., $m_1=xy^{n-1}$, then it follows that $m_2=y^n$. In this case, $\tau_{m'}$ is Lyubeznik-critical for any $m'$.
\end{proof}

\section{Powers of edge ideals that are bridge-friendly}

In this section, we will study graphs $G$ and integers $n$ such that $I(G)^n$ is bridge-friendly. As in the previous section, we begin with the case when $n = 1$.

The following result is immediate from Lemma \ref{lem:existence-m1-m2-m3} and Remark \ref{rem:unique-factor-m1-m2-m3}.

\begin{proposition}\label{prop:not-bridgefriendly-edgeideals}
    An edge ideal $I(G)$ is not bridge-friendly (with respect to $(\succ)$) if and only if  there exist a type-1 set $\tau \subseteq E(G)$ and monomials $m_1\succ m_2 \succ m_3$ in $E(G)$ satisfying the conditions in Lemma \ref{lem:existence-m1-m2-m3}. Moreover, if $m_3=yz$, then in $\tau \cup \{m_1,m_2\}$, $m_1$ and $m_3$ are the only edges containing $y$, and $m_2$ and $m_3$ are the only edges containing $z$, or vice versa. 
\end{proposition}
\begin{proof}
    The proof is straightforward from Lemma \ref{lem:existence-m1-m2-m3} and Remark \ref{rem:unique-factor-m1-m2-m3}.
\end{proof}

Proposition \ref{prop:not-bridgefriendly-edgeideals} roughly says that it is quite restrictive for an edge ideal not to be bridge-friendly. It is known in \cite{CHM24} that edge ideals of trees are bridge-friendly. The next result addresses edge ideals of cycles.

\begin{proposition}\label{prop:non-bridgefriendly-cycles}
    The edge ideal of an $n$-cycle $I(C_n)$ is bridge-friendly if and only if $n\in \{3,5,6\}$.
\end{proposition}

\begin{proof}
    The statement can be verified for $n\leq 6$. Suppose that $n\geq 7$, and set
    \[
    I(C_n)=(x_1x_2,x_2x_3,\dots, x_{n-1}x_n,x_nx_1).
    \]
    Consider any total order $(\succ)$ on $\mingens(I(C_n))$. Without loss of generality, assume $x_1x_2$ is the smallest monomial with respect to $(\succ)$. Set
        $\tau = \{x_1x_2,x_3x_4,x_{n-1}x_{n}\},\ 
        m_1 = x_{2}x_3,\ 
        m_2 = x_{n}x_1,\text{ and }
        m_3 = x_1x_2.$ 
    One can verify that these elements satisfy the condition in Lemma \ref{lem:existence-m1-m2-m3}. Thus, $I(C_n)$ is not bridge-friendly.
\end{proof}

Proposition \ref{prop:non-bridgefriendly-cycles} suggests that, in examining bridge-friendly edge ideals, one should investigate graphs whose induced cycles are only $C_3, C_5$ or $C_6$. Unfortunately, understanding the class of such graphs poses a challenging problem. Our approach is to focus on special classes of graphs for which the cycle structures are better understood. Particularly, we shall consider \emph{chordal} graphs, i.e., graphs whose induced cycles are only triangles.

The following result gives a few additional ``forbidden structures'' for being bridge-friendly. We remark here that the property of being bridge-friendly is purely combinatorial without any reference to the base field, so this can be verified by any computer algebra system. We have made use of \texttt{Sage} \cite{sagemath} to do the exhaustive computations.

\begin{proposition}\label{prop:non-bridgefriendly-graphs}
    Let $G$ be one of the following graphs:\newline
    \begin{minipage}[t]{0.4\textwidth}
    \begin{enumerate}
        \item The complete graph $K_4$ \Kfoursymb.
        \item The gem graph \gemsymb.
    \end{enumerate}
    \end{minipage}%
    \begin{minipage}[t]{0.5\textwidth}
    \begin{enumerate} 
    \setcounter{enumi}{2}
        \item The kite graph \kitesymb.
        \item The net graph \sunletsymb{3}.
    \end{enumerate}
    \end{minipage}
    
    \noindent Then $I(G)$ is not bridge-friendly.
\end{proposition}
\begin{proof}
    Verified with \texttt{Sage} computations.
\end{proof}

Propositions \ref{prop:non-bridgefriendly-cycles} and \ref{prop:non-bridgefriendly-graphs}, coupled with Corollary \ref{cor:induced-subgraph-Morse}, state that graphs that contain any forbidden structures are not bridge-friendly.
The following definition is crucial in classifying chordal graphs that avoid the forbidden structures described in Proposition \ref{prop:non-bridgefriendly-graphs}. 

\begin{definition} \label{def:BF}
    Let $T=(V(T),E(T))$ be a tree graph and let $\w: E(T)\to \mathbb{Z}_{\geq 0}$ be an edge-weight function on $T$. Let $\BF(T,\w)$ denote the graph with vertex set
    \[
    V(T) \sqcup \bigsqcup_{e \in E(T)} \{v_{e,1},v_{e,2},\dots,v_{e,
    \w(e)}\}
    \] 
    and edge set
    \[
    E(T) \sqcup \bigsqcup_{yz = e \in E(T)}  \{yv_{e,i} \}_{i=1}^{\w(e)} \cup \{ zv_{e,i} \}_{i = 1}^{\w(e)}.
    \]
\end{definition}

Roughly speaking, $\BF(T,\w)$ is obtained from $T$ by, for each edge $e \in E(T)$, attaching $\w(e)$ new triangles along the edge $e$. This class of graphs has a nice characterization as follows. 

\begin{proposition}\label{prop:BF-graph-characterization}
    Let $G$ be a graph. Then, $G = \BF(T,\w)$, for a tree graph $T$ and an edge-weight function $\w:E(T)\to \mathbb{Z}_{\geq 0}$, if and only if $G$ is chordal and any induced 3-cycle of $G$ has a vertex of degree 2.
\end{proposition}

\begin{proof}
    Assume that $G=\BF(T,\w)$ for some tree $T$ and function $\w\colon E(T)\to \mathbb{Z}_{\geq 0}$. It is clear that any of its induced 3-cycle has a vertex of degree $2$. It can also be seen that any induced cycle of $G$ is a 3-cycle attached to an edge of $T$. Thus, $G$ is chordal.
    
    We now proceed with the other implication. Let $G$ be a chordal graph such that any of its induced 3-cycle has a vertex of degree 2. Consider the function
    \begin{align*}
        \{\text{induced 3-cycles of } G\}&\to V(G) \times E(G)\\
        C&\mapsto (x_C,y_Cz_C), 
    \end{align*}
    where $x_C$ is a vertex of $C$ of degree 2, and $y_Cz_C$ the edge of $C$ opposite $x_C$. Observe that an induced 3-cycle $C$ of $G$ may have many vertices of degree 2, in which case, we pick $x_C$ to be a fixed one among them. Set 
    \begin{align*}
        V&\coloneqq V(G) \setminus \{x_C \colon C \text{ is an induced 3-cycle of }G\},\\
        E&\coloneqq E(G) \setminus \{x_Cy_C,\ x_Cz_C \colon C \text{ is an induced 3-cycle of }G\}.
    \end{align*}
    Let $T$ be the induced subgraph of $G$ with the vertex set $V$. Since each vertex of $G$ that is not in $T$ is of degree 2, deleting them means deleting the two edges containing it. In other words, $E(T)=E$. 
    
    We claim that for each induced 3-cycle $C$ of $G$, the edge $y_Cz_C$ is an edge of $T$. Indeed, suppose otherwise that $y_C=x_{C'}$ is  a vertex of degree 2 of a different induced 3-cycle $C'$. Then, $C'$ must contain the only two edges containing $y_C$, namely $x_Cy_C$ and $y_Cz_C$, and thus $C'=C$, a contradiction. Therefore, the claim holds and hence $T$ is connected. 
    
    Since any 3-cycle in $G$ becomes an edge in $T$, the latter has no cycle. Thus, $T$ is a tree. We can define the following edge-weight function
    \begin{align*}
        \w: E=E(T)&\to \mathbb{Z}_{\geq 0}\\
        e&\mapsto \card{\{C\colon C\text{ is an induced 3-cycle of } G \text{ such that }y_Cz_C=e\}}.
    \end{align*}
    Since the vertices that we delete from $G$ to obtain $T$ are all of degree 2, $T$ is obtained from $G$ by deleting all of these vertices and all pairs of edges containing each of them. This is a bijective process, i.e., we can obtain $G$ from $T$, provided that we know how many vertices needed to add in for each edge of $T$, which is recorded in the function $\w$. To sum up, $G = \BF(T,\w)$, as desired.
\end{proof}

We are now ready to classify chordal graphs that avoid forbidden structures described in Proposition \ref{prop:non-bridgefriendly-graphs}.

\begin{proposition}\label{prop:BF-characterization-forbidden}
    A chordal graph is $\BF(T,\w)$ for a tree $T$ and an edge-weight function $\w:E(T)\to \mathbb{Z}_{\geq 0}$ if and only if it does not contain, as an induced subgraph any 4-complete graph \Kfoursymb, gem graph \gemsymb, kite grapth \kitesymb, or net graph \sunletsymb{3}.
\end{proposition}

\begin{proof} 
    It is clear that the graph $\BF(T,\w)$ does not contain the list of forbidden graphs, as an induced subgraph. We will only show the other implication. 
    
    Let $G$ be a chordal graph that does not contain any of the listed forbidden induced subgraphs. If $G$ is a tree or $C_3$, then we are done. By Proposition \ref{prop:BF-graph-characterization}, it suffices to show that any induced triangle $C_3$ in $G$ has a vertex of degree 2. Suppose otherwise that $G$ contains a $C_3$, with vertices $x,y,$ and $z$, such that these vertices are of degrees at least 3. Equivalently, $x,y,$ and $z$ are each connected to at least a vertex other than the remaining two vertices. We consider the following possibilities on the additional neighbors of $x,y$ and $z$.
    \begin{itemize}
        \item The vertices $x,y$ and $z$ have a common neighbor outside of the triangle $xyz$. In this case, $G$ contains a 4-complete graph \Kfoursymb, a contradiction. 
        \item The vertices $x,y$ and $z$ have no common neighbor outside of $xyz$, but two of the vertices do. Without loss of generality, assume that for some vertex $w \in V(G)$, $wx, wy \in E(G)$, but $wz \not\in E(G)$. Since $z$ is of degree at least $3$, there exists a neighbor $z'$ of $z$ other than $x$ and $y$. Clearly $z'$ cannot be connected to both $x$ and $y$. Assume that $z'x \not\in E(G)$.
        
        If $z'w \in E(G)$, then $G$ contains an induced 4-cycle on the vertices $z'zxw$, and so $G$ is not chordal. If $z'w \not\in E(G)$ and $z'y \in E(G)$, then $G$ contain an induced gem graph \gemsymb. If $z'x, z'y, z'w \not\in E(G)$, then $G$ contains an induced kite graph \kitesymb. The last two cases both lead to a contradiction.
        \item No two vertices among $\{x,y,z\}$ have a common neighbor outside of the triangle $xyz$. Since the degree of these vertices are at least 3, we may assume that there are additional edges $xx', yy', zz'$, with distinct vertices $x',y',z'$. If there is an edge between $x', y'$ and $z'$, say $x'y' \in E(G)$, then $G$ contains an induced $C_4$ over the vertices $xyy'x'$, and so $G$ is not chordal. If there is no edge between $x', y'$ and $z'$, then $G$ contains an induced net graph \sunletsymb{3}, a contradiction.
        \qedhere
    \end{itemize}
\end{proof}

Before continuing to prove that $\text{BF}(T,\w)$ are bridge-friendly, we shall define a particular total order ($\succ$) on $E(T)$. For a fixed vertex $x_0$ in $T$, we shall view $T$ as a \emph{rooted} tree with root $x_0$. Each vertex $v \in V(T)$ determines a unique path from $v$ to $x_0$. For $i \in \NN$, let 
\begin{equation*} 
	V_i \coloneqq \left\{v \in V(T) ~\middle|~ \dist_T(v,x_0) =i\right\}
\end{equation*} 
be the set of vertices in $T$ whose distance to $x_0$ is $i$. Obviously, $V(T)=\bigcup_{i\in \mathbb{Z}_{\geq 0}} V_i$. Let $c_i = \card{V_i}$, for $i \in \ZZ_{\ge 0}$. We shall consider a  specific labeling for the vertices in $T$ given by writing 
\begin{equation*} 
	V_i = \left\{x_{i,j} ~\middle|~ 1 \le j \le c_i\right\}
\end{equation*}
(with the convention that $x_{0,1} = x_0$.) With respect to this particular labeling of the vertices in $T$, define the following total order $(\succ)$ on $E(T)$:
    \[
    x_{i,j}x_{i+1,k}\succ x_{i',j'}x_{i'+1,k'} \text{ if } \begin{sqcases}
        i<i'; \text{ or}\\
        i=i' \text{ and } j<j';\text{ or} \\
        i=i', j=j' \text{ and } k<k'. 
    \end{sqcases}
    \]
\begin{proposition}\label{prop:bridefriendly-graphs}
    If $G = \BF(T,\w)$, for a tree $T$ and an edge-weight function $\w:E(T)\to \mathbb{Z}_{\geq 0}$, then the edge ideal $I(G)$ is bridge-friendly.
\end{proposition}

\begin{proof} As before, fix a vertex $x_0$ of $T$ and view $T$ as a tree rooted at $x_0$. Let $V_i$, the labeling $x_{i,j}$'s for the vertices of $T$, and the total order ($\succ$) be defined as above. 

We shall first extend the total order $(\succ)$ to $E(G)$. By induction, it suffices to extend the total order ($\succ$) to $E(H)$, where $H$ is obtained by attaching $l$ triangles along an edge $e \in E(T)$. Suppose that $e = \{x_{i,j}, x_{i+1, k}\}$ and the $l$ new vertices in $H$ are $v_{e,1},\dots, v_{e,l}$. Let $e' \succ e$ be the edge immediately before $e$ in the total order ($\succ$) of $E(T)$. The total order on $E(H)$ is now given by setting:
    \[
    e' \succ x_{i,j}v_{e,1}\succ \dots \succ x_{i,j}v_{e,l} \succ x_{i+1,k}v_{e,1} \succ \dots \succ x_{i+1,k}v_{e,l} \succ e.
    \]

We will show that $I(G)$ is bridge-friendly with respect to $(\succ)$. Suppose otherwise that $I(G)$ is not bridge-friendly. By Proposition \ref{prop:not-bridgefriendly-edgeideals}, there exists a collection of edges $\tau \subseteq E(G)$ and edges $m_1\succ m_2 \succ m_3$ in $E(G)$ such that if we set $m_3=yz$, then no other edge in $\tau$ contains $y$ or $z$, and $m_1,m_2$ are gaps of $\tau$. We shall arrive at a contradiction. Consider the following possibilities:
    \begin{itemize}
        \item $m_3$ is an edge of $T$, i.e., $m_3=x_{i,j}x_{i+1,k}$ for some integers $i,j,$ and $k$. The only edges that contain $x_{i+1,k}$, and are larger than $m_3$, are of the form $x_{i+1,k} v_{m_3,l} $ for some $1 \le l \le \w(m_3)$. So, either $m_1$ or $m_2$ must be of this form. Since they are both gaps of $\tau$, and there are exactly two edges that contain $v_{m_3,l}$, we must have $x_{i,j}v_{m_3,l}\in \tau$, a contradiction.
        \item $m_3=x_{i,j}v_{e,l}$ for some edge $e=x_{i,j}x_{i+1,k} \in E(T)$ and integer $l$. Then, the only edge that contains $v_{e,l}$, other than $m_3$ itself, is $x_{i+1,k}v_{e,l}$, which is smaller than $m_3$. Thus, such $m_1$ and $m_2$ do not exist, a contradiction.
        \item $m_3=x_{i+1,k}v_{e,l}$ for some edge $e=x_{i,j}x_{i+1,k} \in E(T)$ and integer $l$. Then, the only edge that contains $v_{e,l}$, other than $m_3$ itself, is $x_{i,j}v_{e,l}$, and the only edges that contain $x_{i+1, k}$, and are bigger than $m_3$ itself, must be of the form $x_{i+1,k}v_{e,l'}$ for some integer $l'$. We can thus assume that $m_1=x_{i,j}v_{e,l}$ and $m_2=x_{i+1,k}v_{e,l'}$. Since $m_2$ is a gap of $\tau$, it follows that $x_{i,j}v_{e,l'}\in \tau$. Observe that $x_{i,j}v_{e,l'}$ and $m_3$ are both in $\tau$ and bigger than $e$, and 
        \[
        e=x_{i,j}x_{i+1,k} \mid \lcm(x_{i,j}v_{e,l'}, m_3).
        \]
        Hence, by Corollary \ref{cor:true-gap}, $e$ is a true gap of $\tau$. This contradicts the fact that $m_2$ can be chosen to be the smallest true gap of $\tau$. \qedhere 
    \end{itemize}
\end{proof}

Combining Propositions \ref{prop:non-bridgefriendly-graphs}, \ref{prop:BF-characterization-forbidden}, and \ref{prop:bridefriendly-graphs}, we obtain our next main result.

\begin{theorem}\label{thm:bridge-friendly-chordal}
    Let $G$ be a chordal graph. The edge ideal $I(G)$ is bridge-friendly if and only if $G$ is  $\BF(T,\w)$ for some tree $T$ and edge-weight function $\w:E(T)\to \mathbb{Z}_{\geq 0}$.
\end{theorem}
 
We end the discussion of bridge-friendly edge ideals with an example of a graph $G$ that does not contain any of the forbidden graphs in Proposition \ref{prop:non-bridgefriendly-graphs}, and yet $I(G)$ is bridge-friendly.

\begin{example}
    Let $G$ be the join of two 6-cycles at three consecutive edges. Note that $G$ is not chordal.

    \begin{center}
    	\joinsixcyclessymb
    \end{center}

    One can verify that $I(G)$ is bridge-friendly with respect to the total order
    \[
    vy \succ ux \succ xz \succ tw \succ yz \succ wz \succ st \succ sv \succ su.
    \]
\end{example}

\begin{remark}
    As in Remark \ref{rm:disconnected}, it is worth noting that Theorem \ref{thm:bridge-friendly-chordal} can be extended to work over disconnected graphs.
\end{remark}

For the remaining of this section, we focus on higher powers of edge ideals of graphs. Similar to what happened to the Lyubeznik resolutions, edge ideals whose some higher power is bridge-friendly form a much smaller class.

\begin{proposition}\label{prop:non-bridgefriendly-powers}
    Let $G$ be one of the following graphs: \newline
    \begin{minipage}[t]{0.4\textwidth}
    \begin{enumerate}
        \item The 4-star graph $K_{1,3}$.
        \item The 4-path graph $P_4$;
        \item The 4-cycle graph $C_4$.
    \end{enumerate}
    \end{minipage}%
    \begin{minipage}[t]{0.5\textwidth}
    \begin{enumerate} 
    \setcounter{enumi}{3}
        \item The paw graph \pawsymb.
        \item The diamond graph \diamondsymb.
        \item The complete graph $K_4$ \Kfoursymb.
    \end{enumerate}
    \end{minipage}
    
    \noindent Then, the ideals $I(G)^2$ and $I(G)^3$ are not bridge-friendly. Moreover, if $G$ is among the first three graphs, then $I(G)^n$ is not bridge-friendly for any $n\geq 2$.
\end{proposition}

\begin{proof}
(1) If $G$ is a 4-star graph $K_{1,3}$, then $I(G)$ behaves in the same way as the ideal generated by three variables. Hence, we can focus on the ideal $J =(x_1,x_2,x_3)$ instead of $I(G)$ and derive at the same conclusions. 

One can verify that $J^2$ and $J^3$ are not bridge-friendly, using \texttt{Sage} computations. Note that for $J^3$, there are only $7!$ total orders that need checking, since the positions of $x_1^3,x_2^3,$ and $x_3^3$ do not matter.
   Consider $n\geq 4$. Compute the HHZ-subideal of $(J^n)^{\leq x_1^nx_2^2x_3^2}$, we get:
    \begin{equation*} 
    	(x_1^n, \ x_1^{n-1}x_2,\ x_1^{n-1}x_3,\ x_1^{n-2}x_2^2,\ x_1^{n-2}x_2x_3,\ x_1^{n-2}x_3^2,\ x_1^{n-3}x_2^2x_3,\ x_1^{n-3}x_2x_3^2,\ x_1^{n-4}x_2^2x_3^2).
    \end{equation*} 
    By factoring out necessary powers of $x_1$, using Proposition \ref{prop:mI}, this ideal behaves the same way as
    \begin{equation*} 
    	(x_1^4, \ x_1^3x_2,\ x_1^3x_3,\ x_1^2x_2^2,\ x_1^2x_2x_3,\ x_1^2x_3^2,\ x_1x_2^2x_3,\ x_1x_2x_3^2,\ x_2^2x_3^2).
    \end{equation*}
    One can verify that this last ideal is not bridge-friendly, and thus neither is $J^n$ nor $I(G)^n$, by Lemma \ref{lem:HHZ-bridge-friendly}. 
    
\noindent (2) If $G$ is the $4$-path graph $P_4$, then one can check that $I(G)^2$ is not bridge-friendly with \texttt{Sage} by exhausting the $6!$ possible total orders. Hence $I(G)^n$ is not bridge-friendly for any $n\geq 2$, by Proposition \ref{prop:HHZ-subideal-edgeideal} and Lemma \ref{lem:HHZ-bridge-friendly}.

\noindent (3) If $G$ is the $4$-cycle graph $C_4$, i.e., $I(G)=(x_1x_2,x_2x_3,x_3x_4,x_1x_4)$, then we can compute the following HHZ-subideal of $I(G)^2$:
        \[
        (I(G)^2)^{\leq (x_1x_2x_3)^2x_4} = (x_1^2x_2^2,\ x_2^2x_3^2,\ x_1^2x_2x_4,\ x_1x_2^2x_3,\ x_2x_3^2x_4,\ x_1x_2x_3x_4).
        \]
Once again, this ideal can be verified to be not bridge-friendly, by checking all $6!$ total orders. Therefore, $I(G)^n$ is not bridge-friendly, for all $n \ge 2$, by Proposition \ref{prop:HHZ-subideal-edgeideal} and Lemma \ref{lem:HHZ-bridge-friendly}.
 
\noindent (4) If $G$ is the paw graph \pawsymb, i.e., $I(G)=(x_1x_2,x_2x_3,x_1x_3,x_3x_4)$, then we can compute the following HHZ-subideal of $I(G)^2$:
        \[
        (I(G)^2)^{\leq (x_1x_3x_4)^2x_2}=(x_1^2x_3^2,\ x_3^2x_4^2,\ x_1^2x_2x_3,\ x_1x_2x_3^2,\ x_1x_3^2x_4,\ x_2x_3^2x_4,\ x_1x_2x_3x_4).
        \]
        One can verify that this ideal is not bridge-friendly by checking all $7!$ total orders. Thus $I(G)^2$ is not bridge-friendly by Lemma \ref{lem:HHZ-bridge-friendly}. Moreover, one can check that
        \[
        (I(G)^3)^{\leq (x_1x_2x_4)^3x_3^2}=x_1x_2I(G)^2,
        \]
        which is the same as $ I(G)^2$. Thus $I(G)^3$ is not bridge-friendly.
        
\noindent (5) If $G$ is the diamond graph \diamondsymb, i.e., $I(G)=(x_1x_2,x_2x_3,x_3x_4,x_1x_4,x_2x_4)$, then we can compute the following HHZ-subideals:
        \begin{align*}
            (I(G)^2)^{\leq (x_2x_4)^2x_1x_3} &= (x_2^2x_4^2,\ x_1x_2^2x_3,\ x_1x_2^2x_4,\ x_1x_2x_4^2,\ x_1x_3x_4^2,\ x_2^2x_3x_4,\ x_2x_3x_4^2,\ x_1x_2x_3x_4),\\
            (I(G)^3)^{\leq (x_2x_4)^3x_1x_3}&=x_2x_4(I(G)^2)^{\leq (x_2x_4)^2x_1x_3}. 
        \end{align*}
        One can verify that $(I(G)^2)^{\leq (x_2x_4)^2x_1x_3}$ is not bridge-friendly by checking all $8!$ total orders. Thus both $I(G)^2$ and $I(G)^3$ are not bridge-friendly by Lemma \ref{lem:HHZ-bridge-friendly}.

\noindent (6) We consider the complete graph $K_4$ on the four vertices $x_1,x_2,x_3,x_4$. One can check that
        \begin{align*}
            (I(K_4)^2)^{\leq (x_2x_4)^2x_1x_3} &=(I(H)^2)^{\leq (x_2x_4)^2x_1x_3}\\
            (I(K_4)^3)^{\leq (x_2x_4)^3x_1x_3}&=x_2x_4(I(K_4)^2)^{\leq (x_2x_4)^2x_1x_3}, 
        \end{align*}
        where $H$ is the diamond graph \diamondsymb \ in part (5). It is then known that $(I(H)^2)^{\leq (x_2x_4)^2x_1x_3}$ is not bridge-friendly and, thus, neither are $I(K_4)^2$ and $I(K_4)^3$ by Lemma \ref{lem:HHZ-bridge-friendly}.\qedhere
\end{proof}

The result and technique for when $G=C_3$ is a bit different, so we will treat it separately.

\begin{proposition}\label{prop:non-bridgefriendly-3cycle}
    The monomial ideal $I(C_3)^n$ is bridge-friendly if and only if $n\leq 3$.
\end{proposition}

\begin{proof}
    Set $I(C_3)=(x_1x_2,x_2x_3,x_1x_3)$. One can check that $I(C_3)^2$ and $I(C_3)^3$ are bridge-friendly with respect to the total orders
    $x_2^2x_3^2 \succ x_1x_2^2x_3 \succ x_1^2x_2^2 \succ  x_1x_2x_3^2 \succ  x_1^2x_3^2
    x_1^2x_2x_3$
    and 
    \[
    x_1^3x_3^3 \succ x_2^3x_3^3 \succ x_1^3x_2^3 \succ x_1^2x_2x_3^3 \succ   x_1x_2^3x_3^2 \succ  x_1^2x_2^3x_3 \succ  x_1^3x_2x_3^2 \succ  x_1^3x_2^2x_3 \succ  x_1^2x_2^2x_3^2 \succ x_1x_2^2x_3^3,  
    \]
    respectively. In fact, we have the following claims:
    \begin{claim}\label{clm:square-3cycle}
        If $(\succ)$ is a total order with respect to which $I(C_3)^2$ is bridge-friendly, then 
        \[
        \min_{\succ} \mingens(I(G)^2)\in \{x_1^2x_2x_3, \ x_1x_2^2x_3,\ x_1x_2x_3^2\}.
        \]
    \end{claim}
    \begin{claim}\label{clm:cube-3cycle}
        If $(\succ)$ is a total order with respect to which $I(C_3)^3$ is bridge-friendly, then 
        \[
        \min_{\succ} \mingens(I(C_3)^3) =x_1^2x_2^2x_3^2.
        \] 
    \end{claim}
    The first claim can be verified by checking all $6!$ total orders. For the second claim, suppose that $(\succ)$ is a total order with respect to which $I(C_3)^3$ is bridge-friendly. By the arguments in the proof of Proposition \ref{prop:HHZ-subideal-edgeideal}, we have
    \begin{align*}
        (I(C_3)^3)^{\leq (x_2x_3)^3x_1^2}=x_2x_3 I(C_3)^2,\\
        (I(C_3)^3)^{\leq (x_1x_3)^3x_2^2}=x_1x_3 I(C_3)^2,\\
        (I(C_3)^3)^{\leq (x_1x_2)^3x_3^2}=x_1x_2 I(C_3)^2.
    \end{align*}
    Combining this observation and Claim \ref{clm:square-3cycle}, if $\min_{\succ} \mingens(I(C_3)^3)$ divides $(x_2x_3)^3x_1^2$, then we must have
    \[
    \min_{\succ} \mingens(I(C_3)^3) \in  \{x_1^2x_2^2x_3^2, \ x_1x_2^3x_3^2,\ x_1x_2^2x_3^3\}.
    \]
    By symmetry, we have similar statements when permuting $1,2,$ and $3$. By checking all 10 possibilities for $\min_{\succ} \mingens(I(C_3)^3)$, the only possibility is $\min_{\succ} \mingens(I(C_3)^3) =x_1^2x_2^2x_3^2$, as claimed.

    Back to the proof of Proposition \ref{prop:non-bridgefriendly-3cycle}. By Proposition \ref{prop:HHZ-subideal-edgeideal} and Lemma \ref{lem:HHZ-bridge-friendly}, it suffices to show that $I(C_3)^4$ is not bridge-friendly. Suppose otherwise that there exists a total order $(\prec)$ with respect to which $I(C_3)^4$ is bridge-friendly. By similar arguments as above, we have
    \begin{align*}
        (I(C_3)^4)^{\leq (x_2x_3)^4x_1^3}=x_2x_3 I(C_3)^3,\\
        (I(C_3)^4)^{\leq (x_1x_3)^4x_2^3}=x_1x_3 I(C_3)^3,\\
        (I(C_3)^4)^{\leq (x_1x_2)^4x_3^3}=x_1x_2 I(C_3)^3.
    \end{align*}
    By Claim \ref{clm:square-3cycle}, if $\min_{\prec} \mingens(I(C_3)^4)$ divides $(x_2x_3)^4x_1^3$, then we must have
    \[
    \min_{\prec} \mingens(I(C_3)^4) =  x_1^2x_2^3x_3^3.
    \]
    By symmetry, we have similar statements when permuting $1,2,$ and $3$. By checking all generators of $I(C_3)^4$, we conclude that no such element exists, a contradiction. This concludes the proof.
\end{proof}

We remark that some of the results above can be generalized, e.g., $I(K_4)^n$ being not bridge-friendly for any $n$ can be shown by the same technique for $I(K_4)^2$. However, when $n\geq 4$, this fact follows directly from Proposition \ref{prop:non-bridgefriendly-3cycle}. 

We are now ready to state the last result of this section. 
\begin{theorem}\label{thm:BF-powers}
    Let $G$ be a graph and let $n\geq 2$ be a positive integer. Then, the ideal $I(G)^n$ is bridge-friendly if and only if one of the following holds:
    \begin{enumerate}
        \item $G$ is an edge or a path of length 2; or
        \item $n=2,3$ and $G$ is a triangle $C_3$.
    \end{enumerate}
\end{theorem}

\begin{proof}
    The ``if" implication follows from \cite[Corollary 5.5]{CK24} and Proposition \ref{prop:non-bridgefriendly-3cycle}. We now show the ``only if" direction.

    Assume that $I(G)^n$ is bridge-friendly. If $n\geq 4$, then by Propositions \ref{prop:non-bridgefriendly-powers} and $\ref{prop:non-bridgefriendly-3cycle}$, $G$ does not contain, as an induced subgraph, any $C_k$, where $k\geq 3$, or $P_k$, where $k\geq 4$. In other words, $G$ must be a star graph $K_{1,k}$ for some integer $k$. By Proposition \ref{prop:non-bridgefriendly-powers} (1), $G$ does not contain $K_{1,3}$ as an induced subgraph. Thus $k\leq 2$, as desired. 

    Now, suppose that $n$ equals to 2 or 3. By Proposition \ref{prop:non-bridgefriendly-powers} (2) and (3),  $G$ is chordal. If $G$ is a tree, then by Proposition \ref{prop:non-bridgefriendly-powers} (1) and (2), $G$ must be $K_{1,1}$ or $K_{1,2}$, as desired. 
    
    It remains to consider the case when $G$ contains, as an induced subgraph, a $C_3$ graph formed by $x,y,z$. If $G=C_3$, then we are done by Proposition \ref{prop:non-bridgefriendly-3cycle}. Suppose that $G \not= C_3$, i.e., there exists a vertex $w \not= x,y,z$ in $G$. Let $N(w)$ denote the set of neighbors of $w$ among $x,y,$ and $z$. We have the following cases:
    \begin{itemize}
        \item $\card{N(w)}=1$, e.g., $N(w)=\{x\}$. In this case, the induced subgraph of $G$ formed by $x,y,z,w$ is a paw graph \pawsymb.
        \item $\card{N(w)}=2$, e.g., $N(w)=\{x,y\}$. The induced subgraph of $G$ formed by $x,y,z,w$ is a diamond graph \diamondsymb.
        \item $\card{N(w)}=3$, i.e., $N(w)=\{x,y,z\}$. The induced subgraph of $G$ formed by $x,y,z,w$ is a complete graph $K_4$ \Kfoursymb.
    \end{itemize}
    We arrive at a contradiction to Proposition \ref{prop:non-bridgefriendly-powers} (4), (5), and (6) in all these cases. This concludes the proof.
\end{proof}


We end the paper with a few questions that we would like to see answered.

\begin{question}
    For which connected graph $G$ is the ideal $I(G)$ bridge-friendly, Barile-Macchia and/or generalized Barile-Macchia (cf. \cite{CHM24, CK24})?
\end{question}


\begin{question}
    For which hypergraph $\H$, is the edge ideal $I(\H)$ Scarf, Lyubeznik, Barile-Macchia, and bridge-friendly?
\end{question}

\bibliographystyle{amsplain}
\bibliography{refs}

\providecommand{\bysame}{\leavevmode\hbox to3em{\hrulefill}\thinspace}
\providecommand{\MR}{\relax\ifhmode\unskip\space\fi MR }
\providecommand{\MRhref}[2]{%
  \href{http://www.ams.org/mathscinet-getitem?mr=#1}{#2}
}
\providecommand{\href}[2]{#2}
\begin{thebibliography}{10}

\bibitem{AFG2020}
Josep \`{A}lvarez Montaner, Oscar Fern\'{a}ndez-Ramos, and Philippe Gimenez, \emph{Pruned cellular free resolutions of monomial ideals}, J. Algebra \textbf{541} (2020), 126--145. \MR{4014733}

\bibitem{BM20}
Margherita Barile and Antonio Macchia, \emph{Minimal cellular resolutions of the edge ideals of forests}, Electron. J. Combin. (2020), P2--41.

\bibitem{BW02}
Ekkehard Batzies and Volkmar Welker, \emph{Discrete {M}orse theory for cellular resolutions}, J. Reine Angew. Math. \textbf{543} (2002), 147--168.

\bibitem{BPS98}
Dave Bayer, Irena Peeva, and Bernd Sturmfels, \emph{Monomial resolutions}, Math. Res. Lett. \textbf{5} (1998), no. 1--2, 31--46.

\bibitem{BS98}
Dave Bayer and Bernd Sturmfels, \emph{Cellular resolutions of monomial modules}, J. Reine Angew. Math. \textbf{502} (1998), 123--140.

\bibitem{CHM24}
Trung Chau, T\`{a}i~Huy H\`{a}, and Aryaman Maithani, \emph{Generalized {B}arile-{M}acchia resolutions for monomial ideals}, in preparation.

\bibitem{CK24}
Trung Chau and Selvi Kara, \emph{Barile--{M}acchia resolutions}, J. Algebraic Combin. \textbf{59} (2024), no.~2, 413--472. \MR{4713508}

\bibitem{CKW24}
Trung Chau, Selvi Kara, and Kyle Wang, \emph{Minimal cellular resolutions of path ideals}, arXiv:2403.16324 [math.AC].

\bibitem{CT2016}
Timothy B.~P. Clark and Alexandre Tchernev, \emph{Regular {CW}-complexes and poset resolutions of monomial ideals}, Comm. Algebra \textbf{44} (2016), no.~6, 2707--2718. \MR{3492183}

\bibitem{CEFMMSS21}
Susan~M Cooper, Sabine El~Khoury, Sara Faridi, Sarah Mayes-Tang, Susan Morey, Liana~M {\c{S}}ega, and Sandra Spiroff, \emph{Simplicial resolutions of powers of square-free monomial ideals}, arXiv:2204.03136.

\bibitem{CEFMMSS22}
\bysame, \emph{Morse resolutions of powers of square-free monomial ideals of projective dimension one}, J. Algebr. Comb. \textbf{55} (2022), no.~4, 1085--1122.

\bibitem{FHHM24}
Sara Faridi, T\`{a}i~Huy H\`{a}, Takayuki Hibi, and Susan Morey, \emph{Scarf complexes of graphs and their powers}, arXiv:2403.05439 [math.AC].

\bibitem{M2}
Daniel~R. Grayson and Michael~E. Stillman, \emph{Macaulay2, a software system for research in algebraic geometry}, Available at \url{https://math.uiuc.edu/Macaulay2/}.

\bibitem{HHZ04}
J\"urgen Herzog, Takayuki Hibi, and Xinxian Zheng, \emph{Dirac’s theorem on chordal graphs and {A}lexander duality}, European J. Combin. \textbf{25} (2004), 949--960.

\bibitem{Ly88}
Gennady Lyubeznik, \emph{A new explicit finite free resolution of ideals generated by monomials in an {R}-sequence}, J. Pure Appl. Alg. \textbf{51} (1988), 193--195.

\bibitem{nauty}
Brendan~D. McKay and Adolfo Piperno, \emph{Practical graph isomorphism, {II}}, J. Symbolic Comput. \textbf{60} (2014), 94--112. \MR{3131381}

\bibitem{OY2015}
Ryota Okazaki and Kohji Yanagawa, \emph{On {CW} complexes supporting {E}liahou-{K}ervaire type resolutions of {B}orel fixed ideals}, Collect. Math. \textbf{66} (2015), no.~1, 125--147. \MR{3295068}

\bibitem{Tay66}
Diana~Kahn Taylor, \emph{Ideals generated by monomials in an {R}-sequence}, Ph.D. thesis, University of Chicago, Department of Mathematics, 1966.

\bibitem{sagemath}
{The Sage Developers}, \emph{{S}agemath, the {S}age {M}athematics {S}oftware {S}ystem ({V}ersion 9.8)}, 2023, {\tt https://www.sagemath.org}.

\bibitem{NK08}
Nicolas Trotignon and Kristina Vuškovic, \emph{{A structure theorem for graphs with no cycle with a unique chord and its consequences}}, Documents de travail du Centre d'Economie de la Sorbonne b08021, Université Panthéon-Sorbonne (Paris 1), Centre d'Economie de la Sorbonne, March 2008.

\bibitem{Vel08}
Mauricio Velasco, \emph{Minimal free resolutions that are not supported by a {CW}-complex}, J. Algebra \textbf{319} (2008), no.~1, 102--114.

\end{thebibliography}
\end{document}